\documentclass{amsart}
\usepackage{a4wide,amssymb,color,tikz}
\usepackage[all]{xy}

\numberwithin{equation}{section}

\newcommand{\IC}{\mathbb C}
\newcommand{\ID}{\mathbb D}
\newcommand{\IT}{\mathbb T}
\newcommand{\IR}{\mathbb R}
\newcommand{\br}{\mathbf r}
\newcommand{\bp}{\mathbf p}
\newcommand{\bu}{\mathbf u}
\newcommand{\bs}{\mathbf s}
\newcommand{\bt}{\mathbf t}

\newcommand{\bh}{\mathbf h}
\newcommand{\bF}{\mathbf r}
\newcommand{\bb}{\mathbf b}
\newcommand{\be}{\mathbf e}

\newcommand{\by}{\mathbf y}
\newcommand{\Ra}{\Rightarrow}
\newcommand{\w}{\omega}

\newcommand{\Supp}{\mathsf{Supp}}
\newcommand{\supp}{\mathsf{supp}}
\newcommand{\zz}{\mathsf{z}}

\newcommand{\IZ}{\mathbb Z}
\newcommand{\II}{\ddot I}
\newcommand{\ii}{\II}
\newcommand{\JJ}{\ddot J}

\newcommand{\Rho}{P}
\newcommand{\Tau}{T}

\newcommand{\e}{\varepsilon}

\newtheorem{theorem}{Theorem}[section]
\newtheorem{lemma}[theorem]{Lemma}
\newtheorem{problem}[theorem]{Problem}

\newtheorem{proposition}[theorem]{Proposition}

\theoremstyle{definition}
\newtheorem{definition}[theorem]{Definition}
\newtheorem{remark}[theorem]{Remark}
\newtheorem{example}[theorem]{Example}
\newtheorem{exercise}[theorem]{Exercise}

\title[Isometries between smooth $2$-dimensional Banach spaces]{Any isometry between the spheres of absolutely smooth $2$-dimensional Banach spaces is linear} 
\author{Taras Banakh} 
\address{Ivan Franko National University of Lviv (Ukraine) and Jan Kochanowski University in Kielce (Poland)}
\email{t.o.banakh@gmail.com}

\begin{document}
\begin{abstract} A 2-dimensional Banach space $X$ is called {\em absolutely smooth} if its unit sphere is the image of the real line under a differentiable function $\br:\mathbb R\to S_X$ whose derivative is locally absolutely continuous and has $\|\br'(s)\|=1$ for all $s\in\IR$. We prove that any isometry $f:S_X\to S_Y$ between the unit spheres of absolutely smooth Banach spaces $X,Y$ extends to a linear isometry $\bar f:X\to Y$ of the Banach spaces $X,Y$. This answers the famous Tingley's problem in the class of absolutely  smooth $2$-dimensional Banach spaces.
\end{abstract}
\maketitle

\section{Introduction}

According to a classical result of Mazur and Ulam \cite{MU} (extended by Mankiewicz \cite{Man}), every bijective isometry $f:B_X\to B_Y$ between the unit balls of Banach spaces $X,Y$ extends to a linear bijective isometry of the spaces $X,Y$.
Tingley \cite{Tingley} asked in 1987 if the unit balls in this Mazur--Ulam--Mankiewicz  result can be replaced by the unit spheres:

\begin{problem}[Tingley]\label{pr:Tingley} Let $f:S_X\to S_Y$ be a bijective isometry between the unit spheres of Banach spaces $X,Y$. Can $f$ be extended to a linear isometry of the Banach spaces $X,Y$?
\end{problem}

Here for a Banach space $(X,\|\cdot\|)$ by $B_X=\{x\in X:\|x\|\le 1\}$ and $S_X=\{x\in X:\|x\|=1\}$ we denote the unit ball and the unit sphere of $X$, respectively.

Motivated by Tingley's Problem, the following property of Banach spaces was introduced by Cheng and Dong \cite{CD}, and then widely used in the literature devoted to Tingley's problem.

\begin{definition} A Banach space $X$ has {\em the Mazur--Ulam property} if every bijective isometry $f:S_X\to S_Y$ from the unit sphere of $X$ to the unit sphere of any Banach space $Y$ extends to a linear isometry between the Banach spaces $X,Y$.
\end{definition}

So, Tingley's Problem~\ref{pr:Tingley} has an affirmative solution if and only if all Banach spaces have the Mazur--Ulam property.  
Many classical Banach spaces do have the Mazur--Ulam property, see  the surveys \cite{Ding}, \cite{Per}, \cite{YZ} and references therein. According to \cite{KM}, every polyhedral finite-dimensional Banach space has the Mazur--Ulam property. Surprisingly, it is still not known if every $2$-dimensional Banach space has the Mazur--Ulam property. The latter problem was discussed in the papers \cite{DZ}, \cite{San}, \cite{Tan}, \cite{WX}. A substantial progress in resolving the $2$-dimensional case of Tingley's problem was made by Javier Cabello S\'anchez who proved in \cite{San}  that an isometry $f:S_X\to S_Y$ of the spheres of $2$-dimensional Banach spaces extends to a linear isometry of $X$ and $Y$ if and only if $f$ is linear on some nonempty relatively open subset $U\subseteq S_X$. The linearity of $f$ on $U$ means that $f(ax+by)=a\cdot f(x)+b\cdot f(y)$ for any $x,y\in U$ and any real numbers $a,b$ with $ax+by\in U$. Generalizing the result of Wang and X.~Huan \cite{WX}, Cabello S\'anchez  also proved in \cite[3.8]{San} that a two-dimensional Banach space $X$ has the Mazur--Ulam property if its unit ball $B_X$ is not strictly convex. Therefore, the $2$-dimensional case of Tingley's  Problem remains open only for strictly convex Banach spaces. It is also open for smooth Banach spaces. Let us recall \cite[p.60]{LT} that a Banach space $X$ is {\em smooth} if for each $x\in S_X$ there exists a unique linear continuous functional $x^*:X\to\IR$ such that $x^*(x)=1=\|x^*\|$. Geometrically this means that the unit ball $B_X$ has a unique supporting hyperplane at $x$.

In this paper we shall present a solution of Tingley's Problem~\ref{pr:Tingley} in a  subclass of smooth $2$-dimensional Banach spaces consisting of absolutely smooth Banach spaces. Absolutely smooth Banach spaces are introduced in Definition~\ref{d:ss} below. But first we need to recall some notions from Real Analysis.

Let $(X,\|\cdot\|)$ be a Banach space and $k$ be a positive integer.  A function $\br :U\to X$ defined on an open subset $U\subseteq\IR$ is
\begin{itemize}
\item {\em of bounded variation} if there exists a real number $E$ such that 
for any points $x_1<y_1<x_2<y_2<\dots<x_n<y_n$ in $U$  we have $\sum_{i=1}^n\|\br(y_i)-\br(x_i)\|\le E$;
\item {\em absolutely continuous} if  for any $\e>0$ there exists $\delta>0$ such that for any points $x_1<y_1<x_2<y_2<\dots<x_n<y_n$ in $U$ with $\sum_{i=1}^n(y_i-x_i)<\delta$ we have $\sum_{i=1}^n\|\br(y_i)-\br(x_i)\|<\e$;
\item {\em locally absolutely continuous} if each point $x\in U$ has an open neighborhood $O_x\subseteq U$ such that the restriction 
$\br{\restriction}O_x$ is absolutely continuous;
\item {\em $C^k$-smooth} if $\br$ has continuous $k$-th derivative $\br^{(k)}$;
\item {\em $AC^{1}$-smooth} if the derivative $\br'$ is locally absolutely continuous.
\end{itemize}

Now we apply these function classes to define some smoothness properties of $2$-dimensional Banach spaces. 

\begin{definition}\label{d:ss} A $2$-dimensional Banach space $(X,\|\cdot\|)$ is defined to be {\em $C^k$-smooth} (resp. {\em $AC^1$-smooth} or else {\em absolutely smooth}) if there exists a $C^k$-smooth (resp. $AC^1$-smooth) surjective map $\br:\IR\to S_X\subseteq X$ such that $\|\br'(s)\|=1$ for all $s\in\IR$.\end{definition}

For any $2$-dimensional Banach space $X$ these smoothness properties relate as follows (for the last equivalence, see Lemma~\ref{l:r}(8)).
$$
\xymatrix{
\mbox{$C^2$-smooth}\ar@{=>}[r]&\mbox{$AC^1$-smooth}\ar@{=}[r]&\mbox{absolutely smooth}\ar@{=>}[r]&\mbox{$C^1$-smooth}\ar@{<=>}[r]&\mbox{smooth}
}
$$


The main result of this paper is the following theorem that yields an affirmative answer to Tingley's problem for absolutely smooth $2$-dimensional Banach spaces.

\begin{theorem}\label{t:main} Each isometry $f:S_X\to S_Y$ between the unit spheres of two absolutely smooth $2$-dimensional Banach spaces extends to a linear isometry $\bar f:X\to Y$ between the Banach spaces $X,Y$.
\end{theorem}


The proof of Theorem~\ref{t:main} essentially uses the absolute smoothness of both Banach spaces. This motivates the following natural problem.

\begin{problem}\label{prob:MU} Has each absolutely smooth  $2$-dimensional Banach space the Mazur--Ulam property? 
\end{problem}

Theorem~\ref{t:main} reduces Problem~\ref{prob:MU} to the following problem.

\begin{problem} Let $f:S_X\to S_Y$ be an isometry between the unit spheres of $2$-dimensional Banach spaces $X$ and $Y$. Does the absolute smoothness of the Banach space $X$ implies the absolute smoothness of the Banach space $Y$?
\end{problem}

Theorem~\ref{t:main} will be proved in Section~\ref{s:main} after long preparatory work made in Sections~\ref{s:p}--\ref{s:rt}. In Section~\ref{s:m} we introduce the notion of a natural parameterization of the unit sphere of a smooth $2$-dimensional Banach space and prove that such a parameterization is unique up to an isometric shift in the domain. In Section~\ref{s:p}, given a $2$-dimensional Banach space with a fixed basis (such Banach spaces are called 2-based), we introduce the polar parameterization $\bp$ of the unit sphere $S_X$ of $X$ and establish some smoothness properties of this parameterization. In Section~\ref{s:n} we modify the polar parameterization $\bp$ of $S_X$ to the natural parameterization $\br$ of $S_X$, which is a parameterization whose tangent vector has length 1 almost everywhere.
 In Section~\ref{s:c} we introduce the radial and tangential curvatures of the sphere $S_X$. Those are functions $\rho$ and $\tau$ such that $\rho\br-\tau\br'+\br''=0$. 
In Theorem~\ref{t:main2}  we prove that the radial and tangential curvatures determine an absolutely smooth 2-dimensional Banach space uniquely up to an insometry, and will characterize 2-dimensional Hilbert spaces as the unique smooth Banach spaces with constant radial and tangential curvatures (equal to $1$ and $0$, respectively). In Section~\ref{s:RT}, for a smooth 2-based Banach space $X$ we introduce the radial and tangential supercurvatures $\Rho$ and $\Tau$, which are continuous functions such that ${-}\Rho{\cdot}\br+\Tau{\cdot}\br'=\br'\circ \varphi$, where $\varphi:\IR\to\IR$ is a continuous non-decreasing function such that $\br'=\br\circ\varphi$. We observe that the function $\frac{\tau}{\rho}$ extends to the continuous function $\psi=\frac\Tau\Rho$, called the {\em quotient curvature} of $X$. In Section~\ref{s:rt} we prove that the metric of the unit sphere of an absolutely smooth 2-based Banach space uniquely determines the radial curvature $\rho$ and  the derivative $\psi'$ of the quotient curvature $\psi=\frac{\Tau}{\Rho}$. Then we apply one result of Floquet Theory on periodic solutions of second order differential equations to recover the quotient curvature $\psi$  from its derivative $\psi'$. Knowing the radial and quotient curvatures we finally calculate  the tangential curvatures $\tau=\rho\cdot\psi$. Thus we conclude that the radial and tangential curvatures of an absolutely smooth 2-based Banach space $X$ can be uniquely recovered from the metric  of the unit sphere. Together with the Uniqueness Theorem~\ref{t:main2} this yields the proof of Theorem~\ref{t:main}, presented in the final Section~\ref{s:main}. 

The referee of the paper recommended to add that the absolute smoothness of a  2-dimensional Banach space $X$ is the minimum requirement needed for recovering the natural parameterization of $X$ from its second derivative, and there is no way to improve the main result using the same ideas.
\smallskip




\section{Some tools from Real Analysis}\label{s:RA}

Throughout the paper we shall exploit some standard tools of Real Analysis. In this section we recall  some facts from the Real Analysis that will be used in the sequel. 

By Lebesgue Theorems \cite[7.1.13 and 7.1.15]{RA},  each locally absolutely continuous function $f:\IR\to X$ is differentiable almost everywhere and $f(b)-f(a)=\int_a^bf'(x)\,dx$ for any $a<b$. Moreover, $\int_a^b|f'(x)|dx<\infty$. So, a locally absolutely continuous function can be recovered from its derivative (which is a locally integrable function). Let us recall that a function $f:\IR\to\IR$ is {\em locally integrable} if it is measurable and $\int_a^b|f(x)|\,dx$ is finite for any $a<b$.

A point $x\in\IR$ is called a {\em Lebesgue point} of a locally integrable function $f:\IR\to\IR$ if $$\lim_{\e\to 0}\frac1\e\int_0^\e|f(x+t)-f(x)|\,dt=0.$$

The following classical result of Lebesgue can be found in \cite[7.1.20 and 7.1.21]{RA}.

\begin{lemma}[Lebesgue]\label{l:Lebesgue} Let $f:\IR\to\IR$ be a locally integrable function.  Almost every point of the real line is a Lebesgue point of $f$, and each Lebesgue point of $f$ is a differentiability point of the function $F(x)=\int_0^x f(t)dt$.
\end{lemma}

We say that a subset $A$ of the real line has {\em full measure} in $\IR$ if the complement $\IR\setminus A$ has Lebesgue measure zero. Lemma~\ref{l:Lebesgue} implies that for any locally integrable function $f:\IR\to\IR$ the set $\Omega_f$ of Lebesgue points of the function $f$ has full measure in the real line.
\smallskip

For a positive real number $L$, a function $f:\IR\to\IR$ is called {\em $L$-periodic} if $f(x+L)=f(x)$ for all $x\in\IR$. We shall need the following (known) property of $L$-periodic functions.

\begin{lemma}\label{l:perint} If a locally integrable function $f:\IR\to\IR$ is $L$-periodic, then for any $a\in\IR$, $$\int_a^{a+L}f(t)dt=\int_0^L f(t)\,dt.$$
\end{lemma} 


\section{A natural parameterization of the sphere of a $2$-dimensional Banach space}\label{s:m}

Let $(X,\|\cdot\|)$ be a $C^1$-smooth $2$-dimensional Banach space. Then the unit  sphere $S_X=\{x\in X:\|x\|=1\}$ of $X$ is the image of the real line under a $C^1$-smooth map $\br:\IR\to S_X$ such that $\|\br'(s)\|=1$ for every $s\in\IR$. Such a  map $\br$ is called {\em a natural parameterization} of $S_X$. In this section we shall show that a natural parameterization is unique up to an isometry of the real line.
To prove this fact we investigate the relation of natural parameterizations to the intrinsic metrics on  connected subsets of $S_X$.

First we recall the necessary information on intrinsic distances. Let $(M,d)$ be a metric space and $\e$ be a positive real number. A sequence of points $x_0,\dots,x_n\in M$ is called an {\em $\e$-chain} in $M$ if $d(x_{i-1},x_i)<\e$ for all $i\in\{1,\dots,n\}$. For any points $x,y\in M$ let
 $$\breve d_\e(x,y)=\inf\Big(\Big\{\sum_{i=1}^nd(x_{i-1},x_i):\mbox{$x=x_0,\dots,x_n=y$ is an $\e$-chain in $M$}\Big\}\cup\big\{\infty\big\}\Big).$$
It is easy to see that the function $\breve d_\e:M\times M\to[0,\infty]$ has the following properties for any points $x,y,z\in M$:
\begin{itemize}
\item $d(x,y)\le\breve d_\e(x,y)$;
\item $\breve d_\e(x,y)=0$ if and only if $x=y$;
\item $\breve d_\e(x,y)=\breve d_\e(y,x)$;
\item $\breve d_\e(x,z)\le\breve d_\e(x,y)+\breve d_\e(y,z)$.
\end{itemize}
In the last item we assume that $r+\infty=\infty=\infty+r$ for any $r\in[0,\infty]$.

For any $x,y\in M$, the finite or infinite number
$$\breve d(x,y)=\sup_{\e>0}\breve d_\e(x,y)$$ is called the {\em intrinsic distance} between the points $x$ and $y$ in the metric space $M$. If $\breve d(x,y)<\infty$ for any points $x,y\in M$, then the intrinsic distance $\breve d$ is called the {\em intrinsic metric} of $M$.

The properties of the distances $\breve d_\e$ for $\e>0$ imply the analogous properties of the intrinsic distance $\breve d$:
\begin{itemize}
\item $d(x,y)\le\breve d(x,y)$;
\item $\breve d(x,y)=0$ if and only if $x=y$;
\item $\breve d(x,y)=\breve d(y,x)$;
\item $\breve d(x,z)\le\breve d(x,y)+\breve d(y,z)$;
\end{itemize}
for any $x,y,z\in M$.


\begin{lemma}\label{l:M} Let $X$ be a $C^1$-smooth $2$-dimensional Banach space and $\br:\IR\to S_X$ be a natural parameterization of its unit sphere. For any closed interval $[a,b]\subseteq \IR$ with $\br([a,b])\ne S_X$, the restriction $\br{\restriction}[a,b]$ is an isometry of $[a,b]$ onto the arc $\br([a,b])\subseteq S_X$ endowed with its intrinsic metric $\breve d$.
\end{lemma}

\begin{proof} The condition $\|\br'\|=1$ implies that $\br:\IR\to S_X$ is a local homeomorphism. Since $\br([a,b])\ne S_X$, the restriction $\br{\restriction}[a,b]$ is injective. To see that it is an isometry, take any points $x<y$ in the segment $[a,b]$. Given any $\e>0$ take any $\e$-chain $x=x_0<x_1<\dots<x_n=y$ in $[a,b]$. Since the map $\br:\IR\to S_X$ is non-expanding (as $\|\br'\|=1$), the sequence $\br(x)=\br(x_0),\dots,\br(x_n)=\br(y)$ is an $\e$-chain in $\br([a,b])$ with $$\sum_{i=1}^n\|\br(x_i)-\br(x_{i-1})\|\le\sum_{i=1}^n|x_i-x_{i-1}|=|x-y|.$$
Then $\breve d_\e(\br(x),\br(y))\le|x-y|$ and hence $\breve d(\br(x),\br(y))\le|x-y|$.
\smallskip

We claim that $\breve d(\br(x),\br(y))=|x-y|$. To derive a contradiction, assume that $\breve d(\br(x),\br(y))<|x-y|$. Then $x\ne y$ and $\breve d(\br(x),\br(y))>0$. Choose a number $\e>0$ such that 
$$
\frac{1+\e}{1-\e}\cdot\breve d(\br(x),\br(y))<|x-y|.
$$ 
By the uniform continuity of the restriction $\br'{\restriction}[a,b]$, there exists $\delta>0$ such that $\|\br'(t)-\br'(s)\|<\e$ for any $t,s\in[a,b]$ with $|t-s|<\delta$. 

By the compactness of $[a,b]$, the injective map $\br{\restriction}[a,b]$ is a topological embedding and by the compactness of $\br([a,b])$, the inverse map $(\br{\restriction}[a,b])^{-1}:\br([a,b])\to [a,b]$ is uniformly continuous. So, there exists $\epsilon>0$ such that $|s-t|<\delta$ for any $s,t\in[a,b]$ with $\|\br(s)-\br(t)\|<\epsilon$.

By the definition of $\breve d(\br(x),\br(y))>0$, there exists a $\epsilon$-chain $\br(x)=y_0,\dots,y_n=\br(y)$ in $\br([a,b])$ such that $\sum_{i=1}^n\|y_{i-1}-y_i\|<(1+\e)\cdot\breve d(\br(x),\br(y))$. For every $i\in\{0,\dots,n\}$ choose a (unique) point $x_i\in [a,b]$ such that $y_i=\br(x_i)$. The choice of $\epsilon$ ensures that $|x_{i-1}-x_i|<\delta$ and hence $\|\br'(t)-\br'(x_i)\|<\e$ for every $t\in [x_{i-1},x_i]$. Observe that
$$
\begin{aligned}
&\|\br(x_i)-\br(x_{i-1})\|=\Big\|\int_{x_{i-1}}^{x_i}\br'(t)dt\Big\|=\Big\|\int_{x_{i-1}}^{x_i}\br'(x_i)dt+\int_{x_{i-1}}^{x_i}(\br'(t)-\br'(x_i))dt\Big\|\ge\\
&\Big\|\int_{x_{i-1}}^{x_i}\br'(x_i)dt\Big\|-\Big\|\int_{x_{i-1}}^{x_i}(\br'(t)-\br'(x_i))dt\Big\|\ge\|\br'(x_i)\|\cdot|x_i-x_{i-1}|-\int_{x_{i-1}}^{x_i}\|\br'(t)-\br'(x_i)\|dt\ge \\
&1\cdot|x_i-x_{i-1}|-\e\cdot|x_i-x_{i-1}|=(1-\e)\cdot|x_i-x_{i-1}|
\end{aligned}
$$and hence
$$
(1+\e)\cdot\breve d(\br(x),\br(y))>\sum_{i=1}^n\|y_i-y_{i-1}\|=\sum_{i=1}^n\|\br(x_i)-\br(x_{i-1})\|>(1-\e)\sum_{i=1}^n|x_i-x_{i-1}|\ge(1-\e)|x-y|.
$$
Then
$$|x-y|< \frac{1+\e}{1-\e}\cdot\breve d(\br(x),\br(y))<|x-y|,$$which is a contradiction completing the proof of the equality $|x-y|=\breve d(\br(x),\br(y))$.
\end{proof}

Now we prove the main result of this section.

\begin{lemma}\label{l:iso} For any natural parameterizations $\br_1,\br_2:\IR\to S_X$ of the unit sphere of a $C^1$-smooth $2$-dimensional Banach space $X$, there exists an isometry $\Phi:\IR\to\IR$ of the real line such that $\br_1=\br_2\circ\Phi$.
\end{lemma}

\begin{proof} For every $k\in\{1,2\}$, the differentiability of $\br_k$ and the equality $\|\br_k'\|=1$ imply that $\br_k:\IR\to S_X$ is a covering map. By \cite[1.30]{Hat}, there exist unique continuous maps $\Phi,\Psi:\IR\to\IR$ such that $\br_1=\br_2\circ\Phi$ and $\br_2=\br_1\circ\Psi$. The uniqueness of liftings \cite[1.30]{Hat} implies that $\Phi\circ\Psi=\Psi\circ\Phi$ is the identity map of $\IR$, which means that $\Phi$ is a homeomorphism of the real line. By the continuity of the map $\br_1$, there exists an increasing sequence of real numbers $(x_n)_{n\in\IZ}$ such that $\IR=\bigcup_{n\in\IZ}[x_n,x_{n+1}]$ and $\br_1([x_n,x_{n+1}])\ne S_X$ for every $n\in\IZ$.  For every $n\in\IZ$ consider the real number $y_n=\Phi(x_n)$ and observe that $(y_n)_{n\in\IZ}$ is a monotone sequence of real numbers in the real line such that $\IR=\bigcup_{n\in\IZ}[y_n,y_{n+1}]$ and $\br_1([x_n,x_{n+1}])=\br_2([y_n,y_{n+1}])$ for every $n\in\IZ$. By Lemma~\ref{l:M}, for every $n\in\IZ$ the maps $\br_1{\restriction}[x_n,x_{n+1}]$ and $\br_2{\restriction}[y_n,y_{n+1}]$ are isometries of the segments $[x_n,x_{n+1}]$ and $[y_n,y_{n+1}]$ onto the arc $\br_1([x_n,x_{n+1}]=\br_2([y_n,y_{n+1}])$ endowed with its intrinsic metric. Then the map $\Phi{\restriction}[x_n,x_{n+1}]=(\br_2{\restriction}[y_n,y_{n+1}])^{-1}\circ (\br_1{\restriction}[x_n,x_{n+1}])$ is an isometry of the interval $[x_n,x_{n+1}]$ onto the interval $[y_n,y_{n+1}]$. Having this information, it is easy to conclude that  the homeomorphism $\Phi$ is an isometry of the real line.
\end{proof}

\begin{remark} Any isometry $\Phi$ of the real line is of the form $\Phi(x)=ax+b$ for some $a\in\{-1,1\}$ and $b\in\IR$.
\end{remark}

\section{The polar parameterization of the unit sphere of a $2$-based Banach space}\label{s:p}


By a {\em $2$-based Banach space} we understand a 2-dimensional real Banach space endowed with a basis. 

Let $(X,\|\cdot\|)$ be a $2$-based Banach space and ${\mathbf e}_1,{\mathbf e}_2$ be the basis of $X$.  Let $\be_1^*,\be_2^*:X\to\IR$ be the biorthogonal functionals to the basis $\be_1,\be_2$, which means that $$\be_1^*(\be_1)=1=\be_2^*(\be_2)\mbox{ \ and \ } \be_1^*(\be_2)=0=\be_2^*(\be_1).$$ On the Banach space $X$ consider the equivalent (Euclidean) norm $|\cdot|$ defined  by $$|x|=\sqrt{|\be_1^*(x)|^2+|\be_2^*(x)|^2}.$$ To compare the norms $\|\cdot\|$ and $|\cdot|$, consider the constants
$$c=\min\{\|x\|:x\in X,\;|x|=1\}\mbox{ \ and \ }C=\max\{\|x\|:x\in X,\;|x|=1\}.$$
For a real number $t$ it will be convenient to denote the element $\cos(t){\mathbf e}_1+\sin(t){\mathbf e}_2$ of $X$ by $\mathbf e^{it}$. It is clear that $|\mathbf e^{it}|=1$.

\begin{definition} The map 
$$\bp:\IR\to S_X,\;\;\bp:t\mapsto\frac{\mathbf e^{it}}{\|\mathbf e^{it}\|},$$is called  {\em the polar parameterization} of the unit sphere $S_X=\{x\in X:\|x\|=1\}$ of the $2$-based Banach space $X$.
\end{definition}

The following lemma establishes the Lipschitz property of the polar parameterization.

\begin{lemma}\label{l:Lip}  
For every $t,\e\in\IR$ we have the lower and upper bounds: 
$$\frac{c}{C}\cdot|\sin(\e)|\le\|\bp(t+\e)-\bp(t)\|\le\frac{4C^2}{c^2}\cdot |\sin(\tfrac\e2)|\le \frac{2C^2}{c^2}|\e|.$$
\end{lemma}

\begin{proof} Observe that
$$
\begin{aligned}
&\|\bp(t+\e)-\bp(t)\|=\Big\|\frac{\mathbf e^{i(t+\e)}}{\|\mathbf e^{i(t+\e)}\|}-\frac{\mathbf e^{it}}{\|\mathbf e^{it}\|}\Big\|=\Big\|\frac{\mathbf e^{i(t+\e)}\|\mathbf e^{it}\|-\mathbf e^{it}\|\mathbf e^{i(t+\e)}\|}{\|\mathbf e^{i(t+\e)}\|\cdot\|\mathbf e^{it}\|}\Big\|=\\
&\Big\|\frac{\mathbf e^{i(t+\e)}\|\mathbf e^{it}\|-\mathbf e^{it}\|\mathbf e^{it}\|+\mathbf e^{it}\|\mathbf e^{it}\|-\mathbf e^{it}\|\mathbf e^{i(t+\e)}\|}{\|\mathbf e^{i(t+\e)}\|\cdot\|\mathbf e^{it}\|}\Big\|\le\\
&\frac{\|\mathbf e^{i(t+\e)}-\mathbf e^{it}\|\cdot \|\mathbf e^{it}\|+\|\mathbf e^{it}\|\cdot \big|\|\mathbf e^{it}\|-\|\mathbf e^{i(t+\e)}\|\big|}{ c^2}\le2\frac{\|\mathbf e^{i(t+\e)}-\mathbf e^{it}\|\cdot \|\mathbf e^{it}\|}{c^2}\le\\
&\frac{2C^2|\mathbf e^{i(t+\e)}-\mathbf e^{it}|}{c^2}=\frac{2C^2|\mathbf e^{i\e}-1|}{c^2}=\frac{4C^2|\sin(\e/2)|}{c^2}\le\frac{2C^2|\e|}{c^2}
\end{aligned}
$$
and
$$
\begin{aligned}
&
\|\bp(t+\e)-\bp(t)\|\ge c\cdot|\bp(t+\e)-\bp(t)|=c\cdot\Big|\frac{\mathbf e^{i(t+\e)}}{\|\mathbf e^{i(t+\e)}\|}-\frac{\mathbf e^{it}}{\|\mathbf e^{it}\|}\Big|\ge\\
&\frac{c}{\|\mathbf e^{it}\|}\min\{|r\mathbf e^{i(t+\e)}-\mathbf e^{it}|:r>0\}\ge \frac{c}{C}\min\{|re^{i\e}-1|:r\ge0\}=\frac{c}{C}\cdot |\sin(\e)|.
\end{aligned}
$$
\end{proof}

Next we establish some differentiability properties of the polar parameterization.

\begin{lemma}\label{l:p} The polar parameterization $\bp:\IR\to S_X$, $\bp:\IR\mapsto\frac{\mathbf e^{it}}{\|\mathbf e^{it}\|}$, has the following properties:
\begin{enumerate}
\item $\bp(t+\pi)=-\bp(t)$ for every $t\in\IR$;
\item the function $\bp$ has one-sided derivatives $$\bp'_-(t)=\lim_{\e\to-0}\frac{\bp(t+\e)-\bp(t)}{\e}\mbox{ \ and \ }\bp'_+(t)=\lim_{\e\to+0}\frac{\bp(t+\e)-\bp(t)}{\e}$$ at each point $t\in\IR$;
\item the set $\Lambda_{\bp}=\{t\in \IR:\bp'_-(t)\ne\bp'_+(t)\}$ is at most countable;
\item the functions $\bp'_-$ and $\bp'_+$ have bounded variation on bounded subsets of $\IR$;
\item the function $\bp$ is twice differentiable almost everywhere and its second derivative $\bp''$ is measurable and locally integrable;
\item $\dfrac{c}{C}\le \min\{\|\bp'_-(t)\|,\|\bp'_+(t)\|\}\le 
 \max\{\|\bp'_-(t)\|,\|\bp'_+(t)\|\}\le \dfrac{2C^2}{c^2}$  for every $t\in\IR$.
\end{enumerate}
\end{lemma}

\begin{proof} 1. Observe that for every $t\in\IR$
$$\bp(t+\pi)=\frac{\mathbf e^{i(t+\pi)}}{\|\mathbf e^{i(t+\pi)}\|}=\frac{-\mathbf e^{it}}{\|-\mathbf e^{it}\|}=-\frac{\mathbf e^{it}}{\|\mathbf e^{it}\|}=-\bp(t).$$

2--4. It suffices to prove the properties (2)--(4) of the function $\bp$ locally, i.e., in a neighborhood of any point $\varphi\in\IR$.  Given a real number $\varphi$, consider the point $\bp(\varphi)\in S_X$ and choose a vector $v\in X$ that is linearly independent with the vector $\bp(\varphi)$.  Observe that the function $$f:\IR\to \IR,\;\;f:t\mapsto \|\bp(\varphi)+tv\|$$is convex and strictly positive. By Theorem 3.7.4 in \cite{RA}, the function $f$ has one sided derivatives $f'_-(x)$ and $f'_+(x)$ at each point $x\in\IR$. Moreover, the functions $f'_-$ and $f'_+$ are non-decreasing and the set $\Lambda_f=\{x\in \IR:f'_-(x)\ne f'_+(x)\}$ is at most countable. By Theorem~6.1.3 \cite{RA}, the non-decreasing functions $f'_-$ and $f'_+$ have bounded variation on each compact subset of $\IR$. Then the function
$$\mathbf g:\IR\to S_X,\;\mathbf g:t\mapsto \frac{\bp(\varphi)+tv}{\|\bp(\varphi)+tv\|}=\frac{\bp(\varphi)+tv}{f(t)}$$
has one-sided derivatives $$
\mathbf g'_-(t)=\frac{v\cdot f(t)-(\bp(\varphi)+tv)\cdot f'_-(t)}{f(t)^2}\mbox{ and }\mathbf g'_+(t)=\frac{v\cdot f(t)-(\bp(\varphi)+tv)\cdot f'_+(t)}{f(t)^2}$$at each point $t\in\IR$. Moreover, the functions $\mathbf g'_-$ and $\mathbf g'_+$ have bounded variation on bounded subsets of $\IR$, see Theorems~6.1.9, 6.1.10, 6.1.11 in \cite{RA}.  
\smallskip

Consider the unit sphere $\IT=\{\mathbf e^{it}:t\in\IR\}$ of the Banach space $(X,|\cdot|)$ and the diffeomorphism $h:(\varphi-\pi,\varphi+\pi)\to \IT\setminus\big\{-\frac{\bp(\varphi)}{|\bp(\varphi)|}\big\},\;h(t)\mapsto \mathbf e^{it}$. Next, consider the diffeomorphic embedding $e:\IR\to\IT$, $e:t\mapsto \frac{\bp(\varphi)+vt}{|\bp(\varphi)+vt|}$ and observe that $h^{-1}\circ e:\IR\to (\varphi-\pi,\varphi+\pi)$ is a diffeomorphism of $\IR$ onto some open interval $(a,b)\subseteq(\varphi-\pi,\varphi+\pi)$ that contains $\varphi$. Then $\delta=e^{-1}\circ h:(a,b)\to \IR$ is a diffeomorphism of $(a,b)$ onto the real line.  Observe that for any $t\in(a,b)$ we have $\bp(t)=\mathbf g\circ \delta(t)$. Now the monononicity of the diffeomorphism $\delta$ and the properties of the function $\mathbf g$ imply that the function $\bp$ has one-sided derivatives $\bp'_-(t)=\mathbf g'_-(\delta(t))\cdot \delta'(t)$ and $\bp'_+(t)=\mathbf g_+'(\delta(t))\cdot \delta'(t)$ at each point $t\in(a,b)$, the set $\{t\in(a,b):\bp'_-(t)\ne\bp'_+(t)\}$ is at most countable and the functions $\bp'_-$ and $\bp'_+$ have bounded variation on each compact subset of the interval $(a,b)$, see Theorem~6.1.11 \cite{RA}. 
\smallskip

5. By Theorem 1.3.1 \cite{RA}, each monotone real function  is differentiable almost everywhere and its derivative is measurable and locally integrable. By Theorem 6.1.15 \cite{RA}, each real function that has bounded variation on bounded sets is the difference of two increasing functions. Consequently, any real-valued function that has  bounded variation on  bounded intervals is differentiable almost everywhere and its derivative is measurable and locally integrable. Applying this conclusion to the coordinate functions of the function $\bp'_-$ and $\bp'_+$ (which have bounded variation on bounded sets by already proved Lemma~\ref{l:p}(4)), we conclude that the functions $\bp'_-$ and $\bp'_+$ are differentiable almost everywhere and their derivatives are measurable and locally integrable. 
Since the functions $\bp'_-$ and $\bp'_+$ coincide almost everywhere, the function $\bp$ is twice differentiable almost everywhere and its second derivative $\bp''$ is measurable and locally integrable.   
\vskip3pt

6. The upper and lower bounds for the norms of the one-sided derivatives $\bp'_-(t)$ and $\bp'_+(t)$ can be easily derived from Lemma~\ref{l:Lip}.
\smallskip
\end{proof}





For two vectors $\mathbf x,\mathbf y\in X$ we write $\mathbf x{\upuparrows}\mathbf y$ if $\mathbf x=\alpha{\cdot}\mathbf y$ for some positive real number $\alpha$.

\begin{lemma}\label{l:ab1} Let $t,u,\alpha,\beta\in\IR$ be real numbers such that $\bp'(t){\upuparrows}\bp(u)$ and $$\frac{\bp'(u)}{|\bp'(u)|}=\alpha\cdot\frac{\bp(u)}{|\bp(u)|}+\beta\cdot\frac{\bp(t)}{|\bp(t)|}.$$ Then $|\alpha|\le \frac{C+c}c{\cdot}|\beta|$.
\end{lemma}

\begin{proof} In the Banach space $X$ consider the basis 
$$\bb_1=\mathbf e^{iu}=\cos(u){\mathbf e}_1+\sin(u){\mathbf e}_2=\tfrac{\bp(u)}{|\bp(u)|}\mbox{ \ and \ }\bb_2=\mathbf e^{i(u+\frac\pi2)}=-\sin(u){\mathbf e}_1+\cos(u){\mathbf e}_2.$$ 
Let $\bb_1^*,\bb_2^*:X\to\IR$ be the biorthogonal functionals to the basis $\bb_1,\bb_2$, which means that
$$\bb_1^*(\bb_1)=1=\bb_2^*(\bb_2)\mbox{ \ and \ }\bb_1^*(\bb_2)=0=\bb_2^*(\bb_1).$$

Let $\ID=\{x\in X:|x|\le 1\}$ be the unit ball in the norm $|\cdot|$. The definitions of the numbers $c=\min_{t\in\IR}\|\mathbf e^{it}\|$ and $C=\max_{t\in\IR}\|\mathbf e^{it}\|$ imply that 
$$\tfrac1{C}\ID\subseteq B_X\subseteq \tfrac1c\ID.$$


By the convexity of the unit ball $B_X$, the tangent line $\bp(u)+\bp'(u)\cdot\IR$ to $B_X$ at the point $\bp(u)\in B_X\subseteq\tfrac1c\ID$ does not intersect the interior of the disk $\tfrac1C\ID\subseteq B_X$. Evaluating the sinus of the angle $\alpha$ between the vectors $\mathbf p'(u)$ and $\mathbf p(u)$ (see Figure~1), we obtain the lower bound:
$$\frac{|\bb_2^*(\bp'(u))|}{|\bp'(u)|}=\sin(\alpha)\ge\frac cC.$$ 

Then
$$
\frac{c}C\le\frac{|\bb_2^*(\bp'(u))|}{|\bp'(u)|}=\big|\bb_2^*(\tfrac{\bp'(u)}{|\bp'(u)|})\big|=\big|\alpha\cdot\bb_2^*\big(\tfrac{\bp(u)}{|\bp(u)|}\big)+\beta\cdot\bb_2^*\big(\tfrac{\bp(t)}{|\bp(t)|}\big)\big|=\big|\beta\cdot\bb_2^*\big(\tfrac{\bp(t)}{|\bp(t)|}\big)\big|\le|\beta|,
$$
which implies $1\le\frac{C}c|\beta|$.
Finally, observe that the equality 
$\tfrac{\bp'(u)}{|\bp'(u)|}=\alpha{\cdot}\tfrac{\bp(u)}{|\bp(u)|}+\beta{\cdot}\tfrac{\bp(t)}{|\bp(t)|}$ implies the desired upper bound $$|\alpha|\le 1+|\beta|\le \frac{C}c{\cdot}|\beta|+|\beta|=\frac{C+c}{c}\cdot|\beta|.$$

\begin{center}
\begin{tikzpicture}
\draw (0,0) circle (1cm);
\draw (0,0) circle (3cm);
\draw (-5,0) -- (5,0);
\draw (0,-4) -- (0,4);
\draw[thick,->] (0,0) -- (2,0) node [anchor = north east] {$\mathbf p(u)$};
\draw (2,0) -- (-3.7,3.8);
\draw(2,0) -- (4.7,-1.8);
\draw[thick,->] (2,0) -- (0.5,1) node [anchor = south west] {$\mathbf p'(u)$};
\node at (-1,-1) {$\frac1C\mathbb D$};
\node at (-2.5,-2.5) {$\frac1c\mathbb D$};
\draw (1.5,0) .. controls (1.5,0.16) .. (1.58,0.27);
\node at (1.3,0.2) {$\alpha$};
\node at (0.1,-4.5) {Figure 1};
\end{tikzpicture}
\end{center}
\end{proof}

\begin{lemma}\label{l:ab2} Let $t,u,\alpha,\beta\in\IR$ be real numbers such that $\bp'(t){\upuparrows}\bp(u)$ and $\bp'(u)=\alpha{\cdot}\bp(u)+\beta{\cdot}\bp(t)$. Then $|\alpha|\le \frac{(C+c)C}{c^2}|\beta|$.
\end{lemma}

\begin{proof} The equality $\bp'(u)=\alpha{\cdot}\bp(u)+\beta{\cdot}\bp(t)$ implies the equality $\frac{\bp'(u)}{|\bp'(u)|}=\alpha\frac{|\bp(u)|}{|\bp'(u)|}\frac{\bp(u)}{|\bp(u)|}+\beta\frac{|\bp(t)|}{|\bp'(u)|}\frac{\bp(t)}{|\bp(t)|}$.   By Lemma~\ref{l:ab1},
$\alpha\tfrac{|\bp(u)|}{|\bp'(u)|}\le\tfrac{C+c}{c}\cdot|\beta|\cdot\tfrac{|\bp(t)|}{|\bp'(u)|}$ and hence $|\alpha|\le\tfrac{C+c}{c}\tfrac{|\bp(t)|}{|\bp(u)|}|\beta|\le \tfrac{C+c}{c}\cdot \tfrac{C}c\cdot|\beta|$.
\end{proof}

\section{The natural parameterization of the unit sphere of a $2$-based Banach space}\label{s:n}

Let $(X,\|\cdot\|)$ be a $2$-based Banach space, ${\mathbf e}_1,{\mathbf e}_2$ be its basis and $$\bp:\IR\to S_X,\;\;\bp:t\mapsto\frac{\mathbf e^{it}}{\|\mathbf e^{it}\|}=\frac{\cos(t){\mathbf e}_1+\sin(t){\mathbf e}_2}{\|\cos(t){\mathbf e}_1+\sin(t){\mathbf e}_2\|}$$ be the polar parameterization of the unit sphere $S_X=\{x\in X:\|x\|=1\}$ of $X$. By Lemmas~\ref{l:Lip} and \ref{l:p}, the function $\bp$ is Lipschitz, has one-sided derivatives $\bp'_-$ and $\bp'_+$ and the set $\Lambda_{\bp}=\{x\in\IR:\bp'_-(t)\ne\bp'_+(t)\}$ is at most countable. Since the derivatives $\bp'_-,\bp'_+$ have bounded norm, they are integrable, so we can consider the continuous increasing function 
$$\bs:\IR\to\IR,\;\;\bs:t\mapsto \int_0^t\|\bp'_{-}(t)\|dt=\int_0^t\|\bp'_+(t)\|dt.$$
For $t\in[0,\pi]$ the value $\bs(t)$ can be thought as the length of the curve on the sphere $S_X$ between the points $\bp(0)$ and $\bp(t)$ in the Banach space $X$ (see Section~\ref{s:m} for a more precise formulation of this assertion). 

The number
$$L=\bs(\pi)=\int_0^\pi\|\bp'_-(t)\|dt=\int_0^\pi\|\bp'_+(t)\|dt$$
is called the {\em half-length} of the sphere $S_X$ in $X$.

The definition of the function $\bs$ and Lemma~\ref{l:p} imply the following lemma describing the smoothness properties of the function $\bs$.

\begin{lemma}\label{l:s} The function $\bs:\IR\to\IR$ has the following properties:
\begin{enumerate}
\item $\bs$ is an increasing Lipschitz function;
\item $\bs$ has one-sided derivatives $\bs'_-(t)=\|\bp'_-(t)\|$ and $\bs'_+(t)=\|\bp'_+(t)\|$ at each point $t\in\IR$;
\item the set $\Lambda_{\bs}=\{t\in\IR:\bs'_-(t)\ne\bs'_+(t)\}$ is at most countable;
\item the functions $\bs'_-$ and $\bs'_+$ have bounded variation on bounded subsets of $\IR$;
\item the function $\bs$ is twice differentiable almost everywhere and its second derivative $\bs''$ is measurable and locally integrable;
\item $\frac{c}{C}\le \min\{\bs'_-(t),\bs'_+(t)\}\le 
\max\{\bs'_-(t),\bs'_+(t)\}\le \frac{2C^2}{c^2}$ for every $t\in\IR$.
\end{enumerate}
\end{lemma}

It follows that the increasing continuous function $\bs(t)$ has an inverse function $\bt(s)$, which is also increasing and continuous, and has the following properties that can be derived from the corresponding properties of the function $\bs$.

\begin{lemma}\label{l:t} The function $\bt:\IR\to\IR$ has the following properties:
\begin{enumerate}
\item $\bt$ is an increasing Lipschitz function;
\item $\bt$ has one-sided derivatives $\bt'_-(s)=\frac1{\bs'_-(\bt(s))}=\frac1{\|\bp'_-(\bt(s))\|}$ and  $\bt'_+(s)=\frac1{\bs'_+(\bt(s))}=\frac1{\|\bp'_+(\bt(s))\|}$ at each point $s\in\IR$;
\item the set $\Lambda_{\bt}=\{s\in\IR:\bt'_-(s)\ne\bt'_+(s)\}$ is at most countable and is equal to the set $\bs(\Lambda_\bs)$ where $\Lambda_\bs=\{t\in\IR:\bs'_-(t)\ne\bs'_+(t)\}$;
\item the functions $\bt'_-=\frac1{\bs'_-\!\circ \bt}$ and $\bt'_+=\frac1{\bs'_+\!\circ\bt}$ have bounded variation on bounded subsets of $\IR$;
\item the function $\bt$ is twice differentiable almost everywhere and its second derivative $\bt''$ is measurable and locally integrable;
\item  $\frac{c^2}{2C^2}\le \min\{\bt'_-(s),\bt'_+(s)\}\le 
\max\{\bt'_-(s),\bt'_+(s)\}\le \frac{C}{c}$ for every $s\in\IR$.
\end{enumerate}
\end{lemma}

Now we can introduce one of central notions in this paper.

\begin{definition} The function $$\br:\IR\to S_X,\;\;\br:s\mapsto \bp(\bt(s)),$$is called  {\em the natural parameterization} of the sphere $S_X$ of the $2$-based Banach space.
\end{definition}

We recall that the number 
$$L:=\bs(\pi)=\int_0^\pi\|\bp'_-(t)\|dt=\int_0^\pi\|\bp'_+(t)\|dt$$stands for  the half-length of the sphere $S_X$.

\begin{lemma}\label{l:r} The natural parameterization $\br:\IR\to S_X$ of $S_X$ has the following properties:
\begin{enumerate}
\item $\br(s+L)=-\br(s)$ for every $s\in\IR$;
\item the function $\br$ has one-sided derivatives $$\br'_-(s)=\lim_{\e\to-0}\frac{\br(s+\e)-\br(s)}{\e}\mbox{ \ and \ }\br'_+(s)=\lim_{\e\to+0}\frac{\br(s+\e)-\br(s)}{\e}$$ at each point $s\in\IR$;
\item the set $\Lambda_{\br}=\{s\in \IR:\br'_-(s)\ne\br'_+(s)\}$ is at most countable;
\item$\br$ is non-expanding and has $\|\br'_-(s)\|=\|\br'_+(s)\|=1$ for every $s\in\IR$;
\item the functions $\br'_-$ and $\br'_+$ have bounded variation on bounded  subsets of $\IR$;
\item the function $\br$ is twice differentiable almost everywhere and its second derivative $\br''$ is measurable and locally integrable.
\item If the Banach space $X$ is absolutely smooth, then the derivative $\br'$ is locally absolutely continuous.
\item The Banach space $X$ is smooth if and only if the map $\br$ is $C^1$-smooth.
\end{enumerate}
\end{lemma}

\begin{proof} 1. By Lemma~\ref{l:p}(1), $\bp(t+\pi)=-\bp(t)$ for all $t\in\IR$, which implies $\bp'_-(t+\pi)=-\bp'_-(t)$ and $\|\bp'_-(t+\pi)\|=\|\bp'_-(t)\|$ for all $t\in\IR$. 
By Lemma~\ref{l:perint}, $\int_{t}^{t+\pi}\|\bp'_-(u)\|du=\int_{0}^{\pi}\|\bp_-'(u)\|\,du=L$ for any $t\in\IR$, which implies 
 $$\bs(t+\pi)=\int_0^{t+\pi}\|\bp_-'(t)\|\,dt=\int_{0}^t\|\bp_-'(t)\|dt+\int_{t}^{t+\pi}\|\bp'_-(t)\|dt=\bs(t)+L.$$ 
 Then for every $s\in \IR$ we have 
 $$\bs(\bt(s)+\pi)=\bs(\bt(s))+L=s+L=\bs(\bt(s+L))$$and $\bt(s+L)=\bt(s)+\pi$, by the injectivity of the function $\bs$. Consequently, $$\br(s+L)=\bp(\bt(s+L))=\bp(\bt(s)+\pi)=-\bp(\bt(s))=-\br(s).$$
 \smallskip
 
 2--6. The properties (2)--(6) follow from the definition of $\br=\bp\circ\bt$ and the corresponding properties of the functions $\bp$ and $\bt$, see Lemmas~\ref{l:p} and \ref{l:t}.
 \smallskip
 
7. If the Banach space $X$ is absolutely smooth, then its sphere admits a natural parameterization $\mathbf f:\IR\to S_X$ of $S_X$ whose derivative $\mathbf f'$ is locally absolutely continuous. By Lemma~\ref{l:iso}, there exists an isometry $\Phi$ of the real line such that $\br=\mathbf f\circ \Phi$. The local absolute continuity of the function $\mathbf f'$ implies the local absolute continuity of the function $\br'=\pm \mathbf f'\circ\Phi$.
\smallskip
 
8. If the Banach space $X$ is smooth, then at each point of the unit shere, the unit ball has a unique supporting hyperplane, which implies that $\br'_-(s)=\br'_+(s)$ for all $s\in\IR$ and hence the function $\br$ is differentiable. The convexity of the unit ball of $X$ implies that locally the curve $\br(\IR)$ coincides with the graph of some convex functions. Taking into account that differentiable convex functions are continuously differentiable (see \cite[3.7.4]{RA}), we conclude that the differentiable function $\br$ is $C^1$-smooth. 

Conversely, the $C^1$-smoothness of the function $\br$ implies that the unit sphere $S_X$ has a unique tangent line at each point, which implies that the unit ball $B_X$ has a unique supporting hyperplane at each point of the unit sphere. The latter means that the Banach space $X$ is smooth. 
\end{proof}

\begin{lemma}\label{l:ab3} Let $s,v,\alpha,\beta\in\IR$ be real numbers such that $\br'(s)=\br(v)$ and $\br'(v)=\alpha{\cdot}\br(v)+\beta{\cdot}\br(s)$. Then $|\alpha|\le \frac{(C+c)C}{c^2}|\beta|$.
\end{lemma}

\begin{proof} Consider the real numbers $t=\bt(s)$ and $u=\bt(v)$ and observe that
$$\bp'(t)=\|\bp'(t)\|\cdot\br'(s)=\|\bp'(t)\|\cdot\br(v)=\|\bp'(t)\|\cdot\bp(\bt(v))=\|\bp'(t)\|\cdot\bp(u)\upuparrows\bp(u)$$ and $$\bp'(u)=\|\bp'(u)\|\cdot\br'(v)=\|\bp'(u)\|{\cdot}\alpha{\cdot}\br(v)+\|\bp'(u)\|{\cdot}\beta{\cdot}\br(s)=\|\bp'(u)\|{\cdot}\alpha{\cdot}\bp(u)+\|\bp'(u)\|{\cdot}\beta{\cdot}\bp(t).$$ Now Lemma~\ref{l:ab2} ensures that
$\|\bp'(u)\|{\cdot}|\alpha|\le \frac{(C+c)C}{c^2}\|\bp'(u)\|{\cdot}|\beta|$. 
Taking into account that $\|\bp'(u)\|\ge \frac{c}{C}>0$ (see Lemma~\ref{l:p}), we conclude that $|\alpha|\le \frac{(C+c)C}{c^2}|\beta|$.
\end{proof}


 
\section{The radial and tangential curvatures of a 2-based Banach space}\label{s:c}

Let $(X,\|\cdot\|)$ be a $2$-based Banach space. By Lemma~\ref{l:r}, the natural parameterization $\br:\IR\to S_X$ of the unit sphere $S_X$ of $X$ is twice differentiable almost everywhere and the second derivative $\br''$ of $\br$ is measurable and locally integrable.

Let $\ddot\Omega_\br$ be the set of parameters $s\in\IR$ at which the derivatives $\br'(s)$ and $\br''(s)$ exist. It is easy to see that $\ddot\Omega_\br$ is a Borel subset in the real line.  As we already know, the set $\ddot\Omega_\br$ has full measure (which means that $\IR\setminus\ddot\Omega_\br$  has Lebesgue measure zero). For every $s\in\ddot\Omega_\br$ the norm $\|\br''(s)\|$ of the vector $\br''(s)$ is called the {\em curvature} of the sphere $S_X$ at the point $\br(s)$. Since the vectors $\br(s)$ and $\br'(s)$ form a basis of the linear space $X$, there are unique real numbers $\rho(s)$ and $\tau(s)$ such that
\begin{equation}\label{eq}
\br''(s)=-\rho(s)\cdot \br(s)+\tau(s)\cdot\br'(s).
\end{equation}
The numbers $\rho(s)$ and $\tau(s)$ are called the {\em radial} and {\em tangential curvatures} of the sphere $S_X$ at the point $\br(s)$. The minus sign in the equation (\ref{eq}) is chosen to make the radial curvature $\rho(s)$ non-negative. 

The functions $\rho$ and $\tau$ are measurable, being solutions of the system of two linear equations with measurable coefficients $\br,\br',\br''$.
To express $\rho$ and $\tau$ via $\br$, $\br'$ and $\br''$, denote by ${\mathbf e}_1^*,{\mathbf e}_2^*$ the biorthogonal functionals of the basis ${\mathbf e}_1,{\mathbf e}_2$ of the Banach space $X$.  
By  Cramer's rule,
\begin{equation}\label{Cramer}
\left\{\begin{gathered}
\rho=\frac{({\mathbf e}_2^*\circ\br'')\cdot({\mathbf e}_1^*\circ \br')-({\mathbf e}_1^*\circ \br'')\cdot({\mathbf e}_2^*\circ\br')}{({\mathbf e}_1^*\circ\br)\cdot({\mathbf e}_2^*\circ\br')-({\mathbf e}_2^*\circ\br)\cdot({\mathbf e}_1^*\circ\br')}\\
\\
\tau=\frac{({\mathbf e}_2^*\circ\br'')\cdot({\mathbf e}_1^*\circ \br)-({\mathbf e}_1^*\circ\br'')\cdot({\mathbf e}_2^*\circ\br)}{({\mathbf e}_1^*\circ\br)\cdot({\mathbf e}_2^*\circ\br')-({\mathbf e}_2^*\circ\br)\cdot({\mathbf e}_1^*\circ\br')}
\end{gathered}\right..
\end{equation}
These formulas imply that the functions $\rho$ and $\tau$ inherit many  properties from the functions $\br,\br',\br''$. In particular, the periodicity of the vector-function $\br,\br',\br''$ implies the periodicity of the functions  $\rho$ and $\tau$.

\begin{lemma} The equalities $\rho(s+L)=\rho(s)$ and $\tau(s+L)=\tau(s)$ hold for all $s\in\ddot\Omega_\br$.
\end{lemma}

\begin{proof} Lemma~\ref{l:r}(1) implies that $$\br(s+L)=-\br(s),\;\;\br'(s+L)=-\br'(s),\mbox{ \ and \ }\br''(s+L)=-\br''(s)$$and hence
\begin{multline*}
\rho(s+L)\cdot \br(s+L)-\tau(s+L)\cdot\br'(s+L)=-\br''(s+L)=\br''(s)=\\ -\rho(s)\cdot\br(s)+\tau(s)\cdot\br'(s)=\rho(s)\cdot\br(s+L)-\tau(s)\cdot\br'(s+L).
\end{multline*}
The linear independence of the vectors $\br(s+L)$ and $\br'(s+L)$ implies that $\rho(s+L)=\rho(s)$ and $\tau(s+L)=\tau(s)$.
\end{proof}

The radial and tangential curvatures determine an absolutely smooth $2$-based Banach space uniquely up to a linear isometry.

\begin{theorem}\label{t:main2} Two absolutely smooth $2$-dimensional Banach spaces $Z,Y$ are linearly isometric if and only if for some bases in $Z,Y$ the corresponding radial and tangential curvatures of the $2$-based Banach spaces $Z,Y$ coincide.
\end{theorem}

\begin{proof} The ``only if'' part is trivial and holds for any (not necessarily absolutely smooth) Banach spaces $Y,Z$. To prove the ``if'' part, assume that the Banach spaces $Y,Z$ have bases such that $\rho_Y=\rho_Z$ and $\tau_Y=\tau_Z$, where $\rho_Y$ and $\tau_Y$ (resp. $\rho_Z$ and $\tau_Z$) are the radial and tangential curvatures of the unit sphere of the $2$-based Banach space $Y$ (resp. $Z$).

Let $\br_Y$ and $\br_Z$ be the natural parameterizations of the spheres $S_Y$ and $S_Z$ of the $2$-based Banach spaces $Y,Z$, respectively. By Lemma~\ref{l:iso}, the absolute smoothness of the Banach spaces $Y,Z$ implies the local absolute continuity of the derivatives $\br'_Y$ and $\br'_Z$. Then the second derivatives $\br''_Y$ and $\br''_Z$ exist almost everywhere and for every $s\ge 0$ we have $$\int_0^s\br''_Y(t)dt=\br'_Y(s)-\br'_Y(0)\mbox{ \ and \ }\int_0^s\br_Z''(t)dt=\br_Z'(s)-\br_Z'(0),$$ see Theorem 7.1.5 in \cite{RA}. 

By the Cramer's formulas (\ref{Cramer}), the continuity of the functions $\br_Y$ and $\br_Y'$ and the local integrability of the function $\br''_Y$ imply the local integrability of the functions  $\rho_Y$ and $\tau_Y$.

Let $F:Y\to Z$ be the linear operator such that $F(\br_Y(0))=\br_Z(0)$ and $F(\br_Y'(0))=\br_Z'(0)$. We claim that $F\circ\br_Y(s)=\br_Z(s)$ for all $s\ge 0$.
It suffices to show that the set $$E=\{s\in\IR:F\circ \br_Y(s)=\br_Z(s)\mbox{ \ and \ }F\circ\br'_Y(s)=\br'_Z(s)\}$$ contains the half-line $[0,\infty)$. Let $a=\sup\{s\in[0,\infty):[0,s]\subseteq E\}$. If $a=\infty$, then $[0,\infty)\subseteq E$ and we are done. So, assume that $a$ is finite. 

By the local integrability of the functions $\rho:=\rho_Y=\rho_Z$ and $\tau:=\tau_Y=\tau_Z$, there exists a positive real number $\e<1$ such that
$$\int_a^{a+\e}(|\rho(t)|+|\tau(t)|)dt<1.$$
Let $$M=\max_{t\in[a,a+\e]}\max\{\|F(\br_Y(t))-\br_Z(t)\|_Z,\|F(\br'_Y(t))-\br'_Z(t)\|_Z\}$$ and $\delta\in[0,\e]$ be a real number such that 
$$\max\{\|F(\br_Y(a+\delta))-\br_Z(a+\delta)\|_Z,\|F(\br'_Y(a+\delta))-\br'_Z(a+\delta)\|_Z\}=M.$$ The definition of the number $a$ implies that $M>0$.

Now observe that
$$
\begin{aligned}
&\|F\circ\br_Y(a+\delta)-\br_Z(a+\delta)\|_Z=\|F\circ\br_Y(a+\delta)-F\circ\br_Y(a)+\br_Z(a)-\br_Z(a+\delta)\|_Z=\\
&\big\|\int_a^{a+\delta}(F(\br'_Y(t))-\br_Z'(t))dt\big\|_Z\le 
\int_a^{a+\delta}\|F(\br'_Y(t))-\br_Z'(t)\|_Zdt\le M\cdot\delta<M,
\end{aligned}
$$
and
$$
\begin{aligned}
&\|F\circ\br'_Y(a+\delta)-\br'_Z(a+\delta)\|_Z=\\
&\big\|\int_a^{a+\delta}(F(\br''_Y(t))-\br_Z''(t))dt\big\|_Z\le \int_a^{a+\delta}\|F(\br''_Y(t))-\br_Z''(t)\|dt=\\
&\int_a^{a+\delta}\|F(-\rho_Y(t)\br_Y(t)+\tau_Y(t)\br_Y'(t))-(-\rho_Z(t)\br_Z(t)+\tau_Z(t)\br_Z'(t))\|_Zdt=\\
&\int_a^{a+\delta}\|-\rho(t)(F(\br_Y(t))-\br_Z(t))+\tau(t)(F(\br_Y'(t))-\br'_Z(t))\|_Zdt\le\\
&\int_a^{a+\delta}\big(|\rho(t)|\cdot\|F(\br_Y(t))-\br_Z(t)\|_Z+|\tau(t)|\cdot\|F(\br_Y'(t))-\br'_Z(t)\|_Z\big)\,dt\le\\
&M\cdot\int_a^{a+\delta}(|\rho(t)|+|\tau(t)|)dt<M,
\end{aligned}
$$
which contradicts the definition of the number $M$.
This contradiction shows that $a=\infty$. Then 
$$F(S_Y)=\{F(\br_Y(s)):s\ge 0\}=\{\br_Z(s):s\ge 0\}=S_Z,$$
which implies that $F$ is a linear isometry of the Banach spaces $X$ and $Y$.
\end{proof}

\begin{example}\label{ex:H} For the complex plane $\IC$, considered as a $2$-based Banach space  with basic vectors ${\mathbf e}_1=1$ and ${\mathbf e}_2=i$ and norm $\|z\|=|z|$, we have $$\br(s)=\bp(s)=e^{is},\;\;\br'(s)=ie^{is},\;\;\br''(s)=-e^{is}=-\br(s)$$ and hence $\rho(s)=1$ and $\tau(s)=0$ for all $s\in\IR$.
\end{example}


\begin{proposition} For a $2$-based Banach space $X$ the following conditions are equivalent:
\begin{enumerate}
\item $X$ is isometric to a Hilbert space;
\item $X$ has radial curvature $\rho(s)=1$ and tangential curvature $\tau(s)=0$ for any $s\in\IR$.
\item $X$ is smooth and has constant radial and tangential curvatures.
\end{enumerate}
\end{proposition}

\begin{proof} $(1)\Ra(2)$ If $X$ is isometric to a Hilbert space, then it is linearly isometric to the Hilbert space $(\IC,|\cdot|)$. By Example~\ref{ex:H}, $\rho(s)=1$ and $\tau(s)=0$ for all $s\in\IR$.
\smallskip

$(2)\Ra(1)$ If $\rho(s)=1$ and $\tau(s)=0$ for all $s\in\IR$, then $\br''(s)=-\br(s)$ is continuous. Then the derivative $\br'$ is continuously differentiable and hence locally absolutely continuous.  By (the proof of) Theorem~\ref{t:main2}, the Banach space $X$ is  linearly isometric to the Hilbert space $(\IC,|\cdot|)$.
\smallskip

$(1)\Ra(3)$ This implication follows from Example~\ref{ex:H}.
\smallskip

$(3)\Ra(1)$ Assume that $X$ is smooth and has constant radial and tangential curvatures $\rho$ and $\tau$. Then the natural parameterization $\br$ of $X$ satisfies the differential equation $\br''-\tau\br'+\rho\br=0$ with constant coefficients, which implies that $\br'$ is continuously differentiable. It is well-known that each $C^2$-smooth solution $z:\IR\to\IC$ of the differential equation $z''-\tau z'+\rho z$ can be written as the linear combination of the functions $e^{\lambda_1s}$, $e^{\lambda_2s}$ where $\lambda_1,\lambda_2$ are the roots of the  equation $\lambda^2-\tau\lambda+\rho=0$. The functions $e^{\lambda_1s}$ and $e^{\lambda_2s}$ are periodic and non-constant if and only if the real part of the complex numbers $\lambda_1,\lambda_2$ is zero, which happens if and only if $\tau=0$ and $\rho>0$. In this case the natural parameterization $\br$ is of the form $\br(s)=x_1e^{i\sqrt{\rho}s}+x_2e^{-i\sqrt{\rho}s}$ for some linearly independent vectors $x_1,x_2\in X$. Now we see that the unit sphere of $X$ coincides with the ellipse $\{x_1e^{i\sqrt{\rho}s}+x_2e^{-i\sqrt{\rho}s}:s\in\IR\}$ and $X$ is isometric to a Hilbert space.
\end{proof}

\begin{remark} Any polyhedral $2$-based Banach space $X$ has radial and tangential curvatures equal to zero everywhere except for finitely many points. Using the Cantor function (which is non-decreasing and non-constant, but has zero derivative almost everywhere, see \cite[1.3.3]{RA}), it is possible to construct two non-isometric smooth two-dimensional Banach spaces whose radial and tangential curvatures are equal to zero almost everywhere. This  shows that the absolute smoothness of the Banach spaces $Y,Z$ in Theorem~\ref{t:main2} cannot be weakened to the smoothness of $Y$ and $Z$. 
\end{remark}

\begin{remark} Theorem~\ref{t:main2} is actually the standard Uniqueness Theorem for solutions of the second order linear differential equation
\begin{equation}\label{eq:Hill}
\br''-\tau{\cdot}\br'+\rho{\cdot}\br=0
\end{equation}  
in a suitable Sobolev space (see \cite{Leoni}). Observe that the natural parameterization $\br$ of the unit sphere $S_X$ of a 2-based Banach space $X$  is $2L$-periodic and the radial and tangential curvatures $\rho$ and $\tau$ are $L$-periodic functions. The problem of the existence of a periodic solution of the  differential equation (\ref{eq:Hill}) with periodic coefficients $\rho$ and $\tau$ is a subject of the Floquet Theory \cite{Floquet} (see also \cite{MW}), which has many applications, in particular, to Celestial Mechanics. One of the conclusions of this theory will be exploited in the proof of Lemma~\ref{l:tau0}.
\end{remark}

\begin{exercise} Calculate the (radial and tangential) curvature of the unit sphere of the $2$-based Banach space $(\IC,\|\cdot\|)$ endowed with the basic vectors $1$ and $i$ and norm $\|x+iy\|=(|x|^p+|y|^p)^{\frac1p}$ for  $p\in(1,\infty)$.
\end{exercise}

\section{Radial and tangential supercurvatures of a Banach space}\label{s:RT}

In this section for any smooth $2$-based Banach space $X$ we introduce two functions $\Rho,\Tau:\IR\to\IR$ which are tightly related to the radial and tangential curvatures of $X$ but have better continuity and differentiability properties.

If the $2$-based Banach space $X$ is $C^1$-smooth, then its natural parameterization $\br$ is continuously differentiable and for every $s\in\IR$ the vector $\br'(s)$ is well-defined and belongs to the unit sphere $S_X$. Then we can find a unique real number $\varphi(s)$ in the interval $(s,s+2L)$ such that $\br'(s)=\br(\varphi(s))$. 
The function $\varphi:\IR\to\IR$ will be called {\em the phase shift} of the natural parameterization $\br$.

Observe that for every $s\in\IR$ the vectors $\br(s)$ and $\br'(s)$ form a basis of the linear space $X$. Consequently, we can write the vector $\br'(\varphi(s))$ as the linear combination
$$\br'(\varphi(s))=-\Rho(s)\cdot\br(s)+\Tau(s)\cdot\br'(s)$$for unique real numbers $\Rho(s)$ and $\Tau(s)$, called the {\em radial and tangential supercurvatures} of the sphere $S_X$ of $X$. 

In the following lemma we establish some properties of the radial and tangential supercurvatures.

\begin{lemma}\label{l:RT} If the $2$-based Banach space $X$ is $C^1$-smooth, then 
\begin{enumerate}
\item The phase shift $\varphi$ is a continuous monotone function.
\item $\Rho$ and $\Tau$ are continuous functions such that $|\Tau(s)|\le\frac{(C+c)C}{c^2}|\Rho(s)|$ and $|\Rho(s)|\ge \frac{c^2}{C^2+Cc+c^2}$ for all $s\in\IR$.
\item If $X$ is absolutely smooth, then the functions $\varphi$, $\Rho$, $\Tau$ and $\frac{\Tau}{\Rho}$ are locally absolutely continuous.
\item For any $s\in\ddot\Omega_{\br}$ we have $\|\br''(s)\|=\varphi'(s)$, $\rho(s)=\Rho(s)\cdot\varphi'(s)$, $\tau(s)=\Tau(s)\cdot\varphi'(s)$, and $|\tau(s)|\le\frac{(C+c)C}{c^2}|\rho(s)|$.
\end{enumerate}
\end{lemma}

\begin{proof} 1. Taking into account that the function $\br'=\br\circ\varphi$ is continuous (by the $C^1$-smoothness of $X$) and the function $\br:\IR\to S_X$ is a local homeomorphism, we conclude that the function $\varphi$ is continuous. The monotonicity of $\varphi$ follows from the convexity of the unit ball $B_X$ (more precisely from the fact that for any real numbers $a<b$ with $\br([a,b])\ne S_X$ the equality $\br'(a)=\br'(b)$ implies that $\br'$ is constant on the inteval $[a,b]$).
\smallskip

2. By Cramer's formulas (resembling (\ref{Cramer})), the functions $\Rho$ and $\Tau$ can be expressed via the continuous functions $\br'\circ\varphi$, $\br$ and $\br'$ and hence are continuous. 

Since $\br'(s)=\br(\varphi(s))$ and $$\br'(\varphi(s))=-\Rho(s){\cdot}\br(s)+\Tau(s){\cdot}\br'(s)=\Tau(s){\cdot}\br(\varphi(s))-\Rho(s){\cdot}\br(s),$$
we can apply Lemma~\ref{l:ab3} to conclude that $|\Tau(s)|\le\frac{(C+c)C}{c^2}|\Rho(s)|$. Then
$$1=\|\br'(\varphi(s))\|\le|\Rho(s)|{\cdot}\|\br(s)\|+|\Tau(s)|{\cdot}\|\br'(s)\|\le|\Rho(s)|+\frac{(C+c)C}{c^2}|\Rho(s)|=\frac{C^2+Cc+c^2}{c^2}\cdot|\Rho(s)|,$$
which implies $|\Rho(s)|\ge\frac{c^2}{C^2+Cc+c^2}$. 
\smallskip

3. If $X$ is absolutely smooth, then by Lemma~\ref{l:r}, the function $\br'=\br\circ\varphi$ is locally absolutely continuous. Using Lemma~\ref{l:M}, it is possible to prove that the local absolute continuity of the function $\br\circ \varphi$ implies the local absolute continuity of the function $\varphi$. By the monotonicity and local absolute continuity of $\varphi$, the composition $\br'\circ \varphi$ is locally absolutely continuous. The functions $\Rho$ and $\Tau$ are locally absolutely continuous since they can be expressed by the Cramer formula via the locally absolutely continuous functions $\br'\circ\varphi$, $\br$ and $\br'$ (see also  Theorem~7.1.10 \cite{RA} for the preservation of absolutely continuous functions by algebraic operations). Since $|\Rho|\ge\frac{c^2}{C^2+Cc+c^2}>0$, the function $\frac{T}{P}$ is well-defined and is locally absolutely continuous.
\smallskip

4. For any $s\in\ddot\Omega_\br$ we get $\br'(s)=\br(\varphi(s))$ and hence $\br''(s)=\br'(\varphi(s))\cdot\varphi'(s)$ and $\|\br''(s)\|=\|\br'(\varphi(s)\|\cdot|\varphi'(s)|=\varphi'(s)$. The linear independence of the vectors $\br(s),\br'(s)$ and the equality
$$-\rho(s)\br(s)+\tau(s)\br'(s)=\br''(s)=\br'(\varphi(s))\cdot\varphi'(s)=\big({-}\Rho(s)\br(s)+\Tau(s)\br'(s)\big)\cdot\varphi'(s)$$imply the equalities $\rho(s)=\Rho(s)\varphi'(s)$ and $\tau(s)=\Tau(s)\cdot\varphi'(s)$. Then $\tau(s)=\Tau(s)\cdot\varphi'(s)=\frac{\Tau(s)}{\Rho(s)}\Rho(s)\cdot\varphi'(s)=\frac{\Tau(s)}{\Rho(s)}\rho(s)$ and hence $|\tau(s)|=\frac{|\Tau(s)|}{|\Rho(s)|}\cdot|\rho(s)|\le\frac{(C+c)C}{c^2}\cdot|\rho(s)|$.
\end{proof}

Lemma~\ref{l:RT} implies that for an (absolutely) smooth $2$-based Banach space $X$, the function $\psi=\frac{\Tau}{\Rho}$ is well-defined and (locally absolutely) continuous. Moreover, if $s\in\ddot\Omega_{\br}$ and $\rho(s)>0$, then $\frac{\tau(s)}{\rho(s)}=\psi(s)$. The function $\psi=\frac{\Tau}{\Rho}$ will be called the {\em quotient curvature} of the shere $S_X$. This function will play a crucial role is the calculations of the tangential curvature $\tau$ via measurements of distances on the sphere $S_X$. Since $\psi=\frac\tau\rho$, the equation~\ref{eq} can be rewritten as
$$\br''=-\rho\cdot\br+\rho\psi\cdot\br'.$$
In the following section we shall use this equation without special reference.

\section{Calculating the radial and tangential curvatures}\label{s:rt}

In this section we derive formulas for calculating the radial and tangential curvatures using measurement of distances on the sphere $S_X$ of an absolutely smooth 2-based Banach space $(X,\|\cdot\|)$.

From now on (and till the end of this section) we assume that $X$ is an absolutely smooth 2-based Banach space and $\br$ is the natural parameterization of $X$. The absolute smoothness of $X$ and Lemma~\ref{l:iso} ensure that the derivative $\br'$ is locally absolutely continuous. By Lemma~\ref{l:Lebesgue}, the locally absolutely continuous function $\br'$ is differentiable almost everywhere and the set $\Omega_{\br''}$ of Lebesgue points of the second derivative $\br''$ is of full measure in the real line. Moreover, any Lebesgue point $s$ of $\br''$ is a differentiability point of $\br'$, which implies that $\Omega_{\br''}\subseteq\ddot\Omega_{\br}$.

\begin{lemma}\label{l:Taylor} For any $s\in\Omega_{\br''}$ and any small $\e$ we have the asymptotic formula
$$\br(s+\e)=\br(s)+\br'(s)\e+\tfrac12\br''(s)\e^2+o(\e^2).$$
\end{lemma}

\begin{proof} Taking into account that $s$ is a Lebesgue point of the function $\br''$, we conclude that 
$$
\Big|\int_0^\e\int_0^t\big(\br''(s+u)-\br''(s)\big)\,dudt\Big|\le\int_0^\e\int_0^t\big|\br''(s+u)-\br''(s)\big|\,dudt=\int_0^\e o(t)dt=o(\e^2).
$$

 Since the functions $\br$ and $\br'$ are absolutely continuous, they can be recovered from their derivatives and hence
 $$
\begin{aligned}
&\br(s+\e)-\br(s)=\int_0^\e\br'(s+t)\,dt=\br'(s)\e+\int_0^\e(\br'(s+t)-\br'(s))\,dt=\\
&=\br'(s)\e+\int_0^\e\int_0^t\br''(s+u)\,dudt=\br'(s)\e+\tfrac12\br''(s)\e^2+\int_0^\e\int_0^t\big(\br''(s{+}u)-\br''(s)\big)\,dudt=\\
&=\br'(s)\e+\tfrac12\br'(s)\e^2+o(\e^2).
\end{aligned}
$$
\end{proof}

The following lemma yields a formula for calculating the radial curvature.

\begin{lemma}\label{l:rho} For any $s\in\Omega_{\br''}$, we have the equality
$$\rho(s)=\lim_{\e\to 0}\frac{2-\|\br(s+\e)-\br(s+L-\e)\|}{\e^2}\ge0.$$
\end{lemma}

\begin{proof} Lemma~\ref{l:Taylor} implies
$$\br(s+\e)+\br(s-\e)=2\br(s)+\br''(s)\e^2+o(\e^2).$$
Taking into account that $\br(s+L)=-\br(s)$, we can observe that
$$
\begin{aligned}
&\br(s+\e)-\br(s+L-\e)=\br(s+\e)+\br(s-\e)=2\br(s)+\br''(s)\e^2+o(\e^2)=\\
&2\br(s)+(-\rho(s)\br (s)+\tau(s)\br'(s))\e^2+o(\e^2)=(2-\rho(s)\e^2)\br(s)+\tau(s)\e^2\br'(s)+o(\e^2)=\\
&(2-\rho(s)\e^2)\big(\br(s)+\frac{\tau(s)\e^2}{2-\rho(s)\e^2}\br'(s)+o(\e^2)\big)=
(2-\rho(s)\e^2)\big(\br(s)+\tfrac12{\tau(s)\e^2}\br'(s)+o(\e^2)\big)=\\
&(2-\rho(s)\e^2)\big(\br(s+\tfrac12\tau(s)\e^2)+o(\e^2)\big). 
\end{aligned}
$$
Then
$$\|\br(s+\e)-\br(s+L-\e)\|=(2-\rho(s)\e^2)\|\br(s-\tfrac12\tau(s)\e^2)+o(\e^2)\|=2-\rho(s)\e^2+o(\e^2)$$and hence
$$\rho(s)=\lim_{\e\to 0}\frac{2-\|\br(s+\e)-\br(s+L-\e)\|}{\e^2}\ge0.$$
\end{proof}

Calculating the tangential curvature is much more tricky and exploits the following series of  lemmas.

 Since $\tau=\rho\cdot \psi$, it suffices to find the function $\psi$. By Lemma~\ref{l:RT}, the quotient curvature $\psi=\frac{\Tau}{\Rho}$ is locally absolutely continuous and hence $\psi$ is differentiable almost everywhere and the derivative $\psi'$ is locally integrable. By Lemma~\ref{l:Lebesgue}, the set $\dot\Omega_{\psi}$ of differentiability points of the function $\psi$ has full measure in the real line.
  

First we derive formulas for calculating the function $\psi'$. Fix any point $s\in\IR$.

Since the function $\br''$ is locally integrable and twice differentiable almost everywhere, the radial curvature $\rho$ is defined almost everywhere and is locally integrable. By Lemma~\ref{l:rho}, the function $\rho$ is non-negative. This facts allow us to consider the following four  functions of a real parameter $\e$:
$$
\begin{aligned}
&I_s(\e)=\int_0^\e\rho(s+u)\,du,
&&\II_s(\e)=\int_0^\e I_s(t)\,dt=\int_0^\e\int_0^t\rho(s+u)\,du\,dt\\
&J_s(\e)=\int_0^\e\rho(s+u)u\,du,
&&\JJ_s(\e)=\int_0^\e J_s(t)\,dt=\int_0^\e\int_0^t\rho(s+u)u\,du\,dt.
\end{aligned}
$$
Observe that
$$\e I_s(\e)=\int_0^\e(u I_s(u))'\,du=\int_0^\e(I_s(u)+u\rho(s+u))du=\II_s(\e)+J_s(\e).$$
It follows that $$|\II_s(\e)|\le|\e I_s(\e)|=o(I_s(\e))\quad\mbox{and}\quad |\JJ_s(\e)|\le|\e J_s(\e)|=o(J_s(\e)).$$

 By the formulas (\ref{Cramer}), the local integrability of the function $\br''$ implies the local integrability of the function $\rho$. By Lemma~\ref{l:Lebesgue}, the set $\Omega_{\rho}$ of Lebesgue points of the function $\rho$ has full measure in the real line.

By analogy with Lemma~\ref{l:Taylor} we can prove the following asymptotic formulas for the functions $I_s$, $J_s$, $\II_s$, $\JJ_s$.

\begin{lemma}\label{l:IJ} If $s$ is a Lebesgue point of the function $\rho$, then
$$
\begin{aligned}
&I_s(\e)=\rho(s)\e+o(\e),&&\II_s(\e)=\tfrac12\rho(s)\e^2+o(\e^2),\\
&J_s(\e)=\tfrac12\rho(s)\e^2+o(\e^2), &&\JJ_s(\e)=\tfrac16\rho(s)\e^3+o(\e^3).
\end{aligned}
$$
\end{lemma}

The following lemma yields an asymptotic formula for the difference $\br'(s+\e)-\br'(s)$.

\begin{lemma}\label{l:r'} If the functions $\br'$ and $\psi$ are differentiable at $s$, then for a small $\e$ the following asymptotic equality holds:
$$
\begin{aligned}
\br'(s+\e)-\br'(s)&=-\br(s)\big(I_s(\e)+\psi(s)\rho(s)J_s(\e)+o(J_s(\e))\big)+\\
&\quad+\br'(s)\big(\psi(s)I_s(\e)+(\psi'(s)-1+\psi(s)^2\rho(s))J_s(\e)+o(J_s(\e))\big).
\end{aligned}
$$
\end{lemma}

\begin{proof} Since $X$ is absolutely smooth, the derivative $\br'$ is locally absolutely continuous and hence 
$$
\begin{aligned}
&\br'(s+\e)-\br'(s)=\int_0^\e\br''(s+t)\,dt=-\int_0^\e\rho(s+t)\br(s+t)\,dt+\int_0^\e\rho(s+t)\psi(s+t)\br'(s+t)\,dt=\\
&=-\int_0^\e\rho(s{+}t)(\br(s){+}\br'(s)t{+}o(t))dt+\int_0^\e\rho(s{+}t)(\psi(s){+}\psi'(s)t{+}o(t))(\br'(s){+}\br''(s)t{+}o(t))\,dt=\\
&=-\br(s)I_s(\e)-\br'(s)J_s(\e)+o(J_s(\e))+\psi(s)\br'(s)I_s(\e)+(\psi(s)\br''(s){+}\psi'(s)\br'(s))J_s(\e)+o(J_s(\e))=\\
&=-\br(s)\big(I_s(\e){+}\psi(s)\rho(s)J_s(\e){+}o(J_s(\e))\big)+\br'(s)\big(\psi(s)I_s(\e){+}(\psi'(s){-}1{+}\psi(s)^2\rho(s))J_s(\e){+}o(J_s(\e))\big).\\
\end{aligned}
$$
\end{proof}

By analogy we obtain an asymptotic formula for the difference $\br(s+\e)-\br(s)$.

\begin{lemma}\label{cl:r} If the functions $\br'$ and $\psi$ are differentiable at $s$, then for  a small $\e$ the following asymptotic equality holds: 
$$
\begin{aligned}
\br(s+\e)-\br(s)&=-\br(s)\big(\II_s(\e)+\psi(s)\rho(s)\JJ_s(\e)+o(\JJ_s(\e))\big)+\\
&\quad+\br'(s)\big(\e+\psi(s)\II_s(\e)+(\psi'(s)-1+\psi(s)^2\rho(s))\JJ_s(\e)+o(\JJ_s(\e))\big).\\
\end{aligned}
$$
\end{lemma}

\begin{proof} Observe that
$$
\begin{aligned}
&\br(s+\e)-\br(s)=\int_{0}^{\e}\br'(s+t)\,dt=\br'(s)\cdot\e+\int_{0}^{\e}\big(\br'(s+t)-\br'(s)\big)\,dt=\\
&=\br'(s)\cdot\e+\int_{0}^\e\int_0^t\br''(s+u)\,du\,dt=\\
&=\br'(s)\,\e+\int_0^\e\int_0^t(-\rho(s+u)\br(s+u)+\rho(s+u)\psi(s+u)\br'(s+u))\,du\,dt=\\
&=\br'(s)\,\e+\int_0^\e\int_0^t(-\rho(s+u)\big(\br(s)+\br'(s)u+o(u)\big)\,du\,dt+\\
&\quad+\int_0^\e\int_0^t\rho(s+u)\big(\psi(s)+\psi'(s)u+o(u)\big)\big(\br'(s)+\br''(s)u+o(u)\big)\,du\,dt=\\
&=\br'(s)\,\e-\br(s)\II_s(\e)-\br'(s)\JJ_s(\e)+o(\JJ_s(\e))+\\
&\quad+\psi(s)\br'(s)\II_s(\e)+(\psi(s)\br''(s)+\psi'(s)\br'(s))\JJ_s(\e)+o(\JJ_s(\e))=\\
&=-\br(s)\big(\II_s(\e)+\psi(s)\rho(s)\JJ_s(\e)+o(\JJ_s(\e))\big)+\\
&\quad+\br'(s)\big(\e+\psi(s)\II_s(\e)+(\psi'(s)-1+\psi(s)^2\rho(s))\JJ_s(\e)+o(\JJ_s(\e))\big).\\
\end{aligned}
$$
\end{proof}

\begin{lemma}\label{l:rho-positive} Assume  $s$ is a Lebesgue point of the function $\rho$ and the functions $\br'$ and $\psi$ are differentiable at $s$. If $\rho(s)>0$, then for small $\e$ we have the asymptotic equality
$$\|\br(s+\e)-\br(s-\e)\|=2\e+(\psi'(s)-1)\rho(s)\e^3+o(\e^3)
$$
implying
$$\psi'(s)=1+\lim_{\e\to 0}\frac{\|\br(s+\e)-\br(s-\e)\|-2\e}{\rho(s)\e^3}.$$
\end{lemma}

\begin{proof} Since the function $\bF$ is continuously differentiable, for a small $\e$ there exists a number $\delta\in(-\e,\e)$ such that 
\begin{equation}\label{eq:ed}
\br(s+\e)-\br(s-\e)=\|\br(s+\e)-\br(s-\e)\|\cdot\br'(s+\delta).
\end{equation}
By Lemmas~\ref{cl:r} and \ref{l:IJ}, $$
\begin{aligned}
&\br(s+\e)-\br(s-\e)=(\br(s+\e)-\br(s))-(\br(s-\e)-\br(s))=\\
&=\br(s)\big(\II_s(\e)-\II_s(-\e)+\psi(s)\rho(s)(\JJ_s(\e)-\JJ_s(-\e))+o(\JJ_s(\e)-\JJ_s(-\e))\big)+\\
&\quad+\br'(s)\big(2\e+\psi(s)(\II_s(\e){-}\II_s({-}\e))+(\psi'(s){-}1{+}\psi(s)^2\rho(s))(\JJ_s(\e){-}\JJ_s({-}\e))+o(\JJ_s(\e){-}\JJ_s({-}\e))\big)=\\
&=-\br(s)\big(\II_s(\e){-}\II_s(-\e)+\tfrac13\psi(s)\rho(s)^2\e^3+o(\e^3)\big)+\\
&\quad+\br'(s)\big(2\e+\psi(s)(\II_s(\e){-}\II_s({-}\e))+\tfrac13(\psi'(s)-1+\psi(s)^2\rho(s))\rho(s)\e^3+o(\e^3)\big).
\end{aligned}
$$
On the other hand, Lemmas~\ref{l:r'} and \ref{l:IJ} imply
$$
\begin{aligned}
\br'(s+\delta)&=-\br(s)\big(I_s(\delta)+\psi(s)\rho(s)J_s(\delta)+o(J_s(\delta))\big)+\\
&\quad+\br'(s)\big(1+\psi(s)I_s(\delta)+(\psi'(s)-1+\psi(s)^2\rho(s))J_s(\delta)+o(J_s(\delta))\big)=\\
&=-\br(s)\big(I_s(\delta)+O(\delta^2)\big)+\br'(s)\big(1+\psi(s)I_s(\delta)+O(\delta^2)\big).
\end{aligned}
$$
Writing the vector equation (\ref{eq:ed}) in coordinates in the basis $\br(s),\br'(s)$, we obtain two equations
\begin{equation}\label{first1}
\II_s(\e)-\II_s(-\e)+\tfrac13\psi(s)\rho^2(s)\e^3+o(\e^3)=\|\br(s+\e)-\br(s-\e)\|\cdot \big(I_s(\delta)+O(\delta^2)\big)
\end{equation}
and
\begin{multline}\label{second2}
2\e+\psi(s)(\II_s(\e)-\II_s(-\e))+\tfrac13(\psi'(s)-1+\psi(s)^2\rho(s))\rho(s)\e^3+o(\e^3)=\\
=\|\br(s+\e)-\br(s-\e)\|\cdot\big(1+\psi(s)I_s(\delta)+O(\delta^2)\big).
\end{multline}

Now we shall derive asymptotic formulas for $\|\br(s+\e)-\br(s-\e)\|$ improving the precision of the expansions in the following five steps.
\smallskip

1. The equation (\ref{second2}) implies that $2\e(1+o(1))=\|\br(s+\e)-\br(s-\e)\|\cdot (1+o(1))$ and hence $\|\br(s+\e)-\br(s-\e)\|=2\e(1+o(1))$. So, we can write $\|\br(s+\e)-\br(s-\e)\|=2\e+n(\e)$ where $n(\e)=o(\e)$.
\smallskip

2. Observe that $$I_s(\delta)+O(\delta^2)=\rho(s)\delta+o(\delta)=\rho(s)\delta(1+o(1))=I_s(\delta)(1+o(1)).$$
After substitution of $\|\br(s+\e)-\br(s-\e)\|=2\e(1+o(1))$ into the equation (\ref{first1}), we obtain that
$$
\begin{aligned}
I_s(\delta)&=\tfrac1{2\e}\big(\II_s(\e)-\II_s(-\e)+O(\e^3)\big)(1+o(1))=\\
&=\tfrac1{2\e}\big((\tfrac12\rho(s)\e^2+o(\e^2))-(\tfrac12\rho(s)(-\e)^2+o(\e^2))+O(\e^3)\big)(1+o(1))=o(\e)
\end{aligned}
$$and hence $\rho(s)(1+o(1))\delta=I_s(\delta)=o(\e)$ and $\delta=o(\e)$.
\smallskip

3. After substitution of $I_s(\delta)$ into the equation (\ref{second2}), we obtain that 
$$
\begin{aligned}
&2\e+\psi(s)(\II_s(\e)-\II_s(-\e))+O(\e^3)=\\
&=(2\e+n(\e))\big(1+\tfrac{\psi(s)(1+o(1))}{2\e}(\II_s(\e)-\II_s(-\e)+O(\e^3)\big)=\\
&=2\e+n(\e)+\psi(s)(1+o(1))(\II_s(\e)-\II_s(-\e)+O(\e^3))=\\
&=2\e+n(\e)+\psi(s)(\II_s(\e)-\II_s(-\e))+o(\e^2)
\end{aligned}
$$
and hence $n(\e)=o(\e^2)$.

4. After substitution of $n(\e)=o(\e^2)$ and $\delta=o(\e)$ into the equation (\ref{first1}), we obtain that $$\II_s(\e)-\II_s(-\e)+\tfrac13\psi(s)\rho(s)^2\e^3+o(\e^3)=
(2\e+o(\e^2))(I_s(\delta)+o(\e^2))$$and hence
$$I_s(\delta)=\frac{1+o(\e^2)}{2\e}\big(\II_s(\e)-\II_s({-}\e)+\tfrac13\psi(s)\rho(s)^2\e^3+o(\e^3)\big)=\tfrac1{2\e}(\II_s(\e)-\II_s({-}\e))+\tfrac16\psi(s)\rho(s)^2\e^2+o(\e^2).$$
\smallskip

5. Finally, put this $I_s(\delta)$  and also $n(\e)=o(\e^2)$ and $\delta=o(\e)$ into the equation (\ref{second2}) and obtain the equation
$$
\begin{aligned}
&2\e+\psi(s)(\II_s(\e)-\II_s(-\e))+\tfrac13(\psi'(s)-1+\psi(s)^2\rho(s))\rho(s)\e^3+o(\e^3)=\\
&=(2\e+n(\e))\cdot 
\big(1+\tfrac{\psi(s)}{2\e}(\II_s(\e)-\II_s(-\e))+\tfrac16\psi(s)^2\rho(s)^2\e^2+o(\e^2))+O(\delta^2)\big)=\\
&=2\e+n(\e)+\psi(s)(\II_s(\e)-\II_s(-\e))+\tfrac13\psi(s)^2\rho(s)^2\e^3+o(\e^3)
\end{aligned}
$$
implying the desired equation
$$\|\br(s+\e)-\br(s-\e)\|=2\e+n(\e)=2\e+\tfrac13(\psi'(s)-1)\rho(s)\e^3+o(\e^3).
$$
\end{proof}

Next, we show how to calculate the derivative $\psi'(s)$ at points $s\in\Omega_\rho$ with $\rho(s)=0$. 

A point $s\in\IR$ is defined to be 
\begin{itemize}
\item {\em $I$-null} if $I_s(\e)=0=I_s(-\e)$ for some non-zero $\e$;
\item {\em $I$-positive} if $I_s(\e)\ne 0$ for all nonzero $\e$.
\end{itemize}
It is easy to see that each $I$-null point is a Lebesgue point of the function $\rho$.

\begin{lemma}\label{l:I-null} If a point $s\in\IR$ is $I$-null, then $\psi'(s)=1$.
\end{lemma}

\begin{proof} Since $s$ is $I$-null, there exists $\e>0$ such that $\int_0^\e\rho(s+t)\,dt=0=\int_0^\e\rho(s-t)\,dt$. Since the function $\rho$ is non-negative this implies that $\rho$ is equal to zero in the interval $(s-\e,s+\e)$. Then $\tau=\rho\cdot\psi$ is also equal to zero in $(s-\e,s+\e)$ and so is $\br''=-\rho\br+\tau\br'$. Since $\br'$ is absolutely continuous, $\br'$ is constant in the interval $(s-\e,s+\e)$. Since $\br'(s+t)=\br(\varphi(s+t))$, the shift $\varphi$ is constant in the interval $(s-\e,s+\e)$.
Then for every $t\in(-\e,\e)$ we have
$$\br'(\varphi(s))=\br'(\varphi(s+t))={-}P(s+t)\br(s+t)+P(s+t)\psi(s+t)\br'(s+t).$$Taking the derivative at $s$, we obtain the equality
$$
\begin{aligned}
0&=\frac d{dt}\big({-}P(s+t)\br(s+t)+P(s+t)\psi(s+t)\br'(s+t)\big)=\\
&={-}P'(s)\br(s)-P(s)\br'(s)+(P'(s)\psi(s)+P(s)\psi'(s))\br'(s)\big),
\end{aligned}
$$
which implies
$P'(s)=0$ and $P(s)(\psi'(s)-1)=0$. By Lemma~\ref{l:RT}, $P(s)>0$ and hence $\psi'(s)=1$.
\end{proof}

Finally, we consider the case of an $I$-positive point $s\in\IR$. In this case the function $I_s(\e)=\int_0^\e\rho(s+u)\,du$ is strictly positive for positive $\e$ and the function $\II_s(\e)=\int_0^\e I_s(t)\,dt$ is strictly increasing on the interval $(0,\infty)$. Also observe that for any positive $\e$ we get $$\II_s(-\e)=\int_0^{-\e}\int_0^t\rho(s+u)\,du dt=\int_0^\e\int_0^t\rho(s-u)\,dudt,$$
which implies that the function $\II_s$ is strictly decreasing on the interval $(-\infty,0)$. 

Then for any sufficiently small $\e>0$ there are unique numbers $\ii_s^-[\e]<0$ and $\ii_s^+[\e]>0$ such that $$\II_s(\ii_s^-[\e])=\e=\II_s(\ii^+_s[\e]).$$
 The continuity and strict monotonicity of the function $\II_s$ on the intervals $(-\infty,0)$ and $(0,\infty)$ impy that the numbers $\ii^-_s[\e]$ and $\ii^+_s[\e]$ are of order $o(1)$ (so they tend to zero as $\e\to+0$).
 
Observe that $\JJ_s(\ii^-_s[\e])<0<\JJ_s(\ii^+_s[\e])$ and hence the difference $\JJ_s(\ii^+_s[\e])-\JJ_s(\ii^-_s[\e])$ is strictly positive.

\begin{lemma}\label{l:I-positive} Assume that $s$ is a Lebesgue point of the function $\rho$ and the functions $\br'$ and $\psi$ are differentiable at $s$. If $\rho(s)=0$ and $s$ is $I$-positive, then
$$\psi'(s)=1+\lim_{\e\to+0}\frac{\|\br(s+\ii_s^+[\e])-\br(s+\ii^-_s[\e])\|-(\ii_s^+[\e]-\ii_s^-[\e])}{\JJ_s(\ii^+_s[\e])-\JJ_s(\ii^-_s[\e])}.$$
\end{lemma}

\begin{proof} Given any sufficiently small number $\e>0$, consider the numbers $\e_-=\II^-_s[\e]<0$ and $\e_+=\II^+_s[\e]>0$. As we already know the numbers $\e_-,\e_+$ and $$\e_\pm:=\e_+-\e_-$$ are of order $o(1)$.

Taking into account that $\rho(s)=0$, $\II_s(\e_-)=\e=\II_s(\e_+)$ and $\JJ_s(\e_-)<0<\JJ_s(\e_+)$, we can apply Lemma~\ref{cl:r} and conclude that
$$
\begin{aligned}
\br(s+\e_+)-\br(s+\e_-)&=-\br(s)\big(o(\JJ_s(\e_+)-\JJ_s(\e_-))\big)+\\
&\quad+\br'(s)\big(\e_+-\e_-+(\psi'(s)-1+o(1))(\JJ_s(\e_+)-\JJ_s(\e_-))\big).
\end{aligned}
$$
Find $\delta\in (\e_-,\e_+)$ such that 
\begin{equation}\label{eq:pm}
\br(s+\e_+)-\br(s+\e_-)=\|\br(s+\e_+)-\br(s+\e_-)\|\cdot\br'(s+\delta).
\end{equation}
By Lemma~\ref{l:r'},
$$\br'(s+\delta)=-\br(s) I_s(\delta)(1+o(1))+\br'(s)\big(1+(\psi(s)+o(1))I_s(\delta)\big).$$
Writing the vector equation (\ref{eq:pm}) in coordinates we obtain two equations
\begin{equation}\label{first-pm}
o(\JJ_s(\e_+)-\JJ_s(\e_-))=\|\br(s+\e_+)-\br(s+\e_-)\|\cdot I_s(\delta)
\end{equation}
and
\begin{equation}\label{second-pm}
\e_\pm+(\psi'(s)-1+o(1))(\JJ_s(\e_+)-\JJ_s(\e_-))=\|\br(s+\e_+)-\br(s+\e_-)\|\cdot\big(1+(\psi(s)+o(1))I_s(\delta)\big).
\end{equation}

Now we shall derive asymptotic formulas for the norm $\|\br(s+\e_+)-\br(s+\e_-)\|$ improving the precision of the expansions in the following three steps.
\smallskip

1. The equation (\ref{second-pm}) implies that $\e_\pm(1+o(1))=\|\br(s+\e_+)-\br(s+\e_-)\|\cdot (1+o(1))$ and hence $\|\br(s+\e_+)-\br(s+\e_-)\|=\e_\pm(1+o(1))$. So, we can write $\|\br(s+\e_+)-\br(s+\e_-)\|=\e_\pm+n(\e)$ where $n(\e)=o(\e_\pm)$.
\smallskip

2. After substitution of $\|\br(s+\e_+)-\br(s+\e_-)\|=\e_\pm(1+o(1))$ into the equation (\ref{first-pm}), we obtain that
$$
I_s(\delta)=\tfrac1{\e_\pm}o(\JJ_s(\e_+)-\JJ_s(\e_-)).
$$
\smallskip

3. After substitution of $I_s(\delta)$ into the equation (\ref{second-pm}), we obtain that 
\begin{multline*}
\e_\pm+(\psi'(s)-1+o(1))(\JJ_s(\e_+)-\JJ_s(\e_-))=\\
=(\e_\pm+n(\e))\big(1+\tfrac{\psi(s)+o(1)}{\e_\pm}o(\JJ_s(\e_+)-\JJ_s(\e_-))\big)=\\
=\e_\pm+n(\e)+o(\JJ_s(\e_+)-\JJ_s(\e_-)),
\end{multline*}
which implies that $n(\e)=(\psi'(s)-1+o(1))(\JJ_s(\e_+)-\JJ_s(\e_-))$. Then
$$\|\br(s+\e_+)-\br(s+\e_-)\|=(\e_+-\e_-)+(\psi'(s)-1+o(1))(\JJ_s(\e_+)-\JJ_s(\e_-))$$and finally,
$$\psi'(s)=1+\lim_{\e\to+0}\frac{\|\br(s+\e_+)-\br(s+\e_-)\|-(\e_+-\e_-)}{\JJ_s(\e_+)-\JJ_s(\e_-)}.$$
\end{proof}

\begin{lemma}\label{l:rho>0} $\displaystyle\int_0^{L}\rho(s)ds>0.$
\end{lemma}

\begin{proof} Assuming that $\int_0^{L}\rho(s)ds=0$ and taking into account that $\rho(s)\ge 0$ almost everywhere, we conclude that $\rho(s)=0$ for almost all $s\in[0,L]$.
Taking into account that $\rho(s+L)=\rho(s)$ for every $s\in\ddot\Omega_{\br}$, we conclude that $\rho(s)$ is zero almost everywhere.  
By Lemma~\ref{l:RT}, $|\tau(s)|\le\frac{(C+c)C}{c^2}\rho(s)=0$ and hence $\tau(s)=0$ almost everywhere. Then also $\br''=-\rho\br+\tau\br'$ is zero almost everywhere and by the local absolute continuity of $\br'$, the function $\br'$ is constant and $\br$ is contained in a line, which is not true. This contradiction shows that $\int_0^{L}\rho(s)\,ds>0$.
\end{proof}

\begin{lemma}\label{l:tau0}$\int_0^{L}\tau(s)\,ds=0.$
\end{lemma}

\begin{proof} Let $\be_1^*,\be_2^*:X\to\IR$ be the biorthogonal linear functionals to the basis $\be_1,\be_2$, which means that $\be_1^*(\be_1)=1=\be_2^*(\be_2)$  and $\be_1^*(\be_2)=0=\be_2^*(\be_1)$. 

Consider the functions $x=\be_1^*\circ \br$ and $y=\be_2^*\circ\br$, and observe that $\br=x{\cdot}\be_1+y{\cdot}\be_2$ and $\br'=x'{\cdot}\be_1+y'{\cdot}\be_2$. The linear independence of the vectors $\br$ and $\br'$ implies that the Wronskian $$W=\begin{vmatrix}x&y\\x'&y'\end{vmatrix}=xy'-x'y$$ of the functions $x,y$ does not take the value zero.  The absolute smoothness of the Banach space $X$ and Lemma~\ref{l:iso} imply the local absolute continuity of the vector-function $\br'$, its  coordinate functions $x'$ and $y'$, the Wronskian $W$ and its logarithm $\ln |W|$.

Consider the derivative of $W$:
$$W'=(xy'-x'y)'=xy''-x''y=x(-\rho y+\tau y')-(-\rho x+\tau x')y=\tau\cdot W.$$
Since $W$ is nowhere equal to zero, this implies $(\ln |W|)'=\frac{W'}{W}=\tau$ and by the local absolute continuity of $\ln |W|$, $$
\begin{aligned}
\int_0^{L}\tau(x)\,dx=\;&\ln |W(L)|-\ln |W(0)|=\\
&\ln|x(L)y'(L)-x'(L)y(L)|-\ln|x(0)y'(0)-x'(0)y(0)|=\\
&\ln|(-x(0)(-y'(0))-(-x'(0))(-y(0))|-\ln|x(0)y'(0)-x'(0)y(0)|=0.
\end{aligned}
$$
\end{proof}

Now we can compose all pieces together and derive a formula for calculating the tangential curvature $\tau$ at almost all points of the real line. The definition of an $I$-null point implies that the set $Z_\rho$ of all $I$-null points is open in the real line and $Z_\rho\subseteq\Omega_\rho$. It follows that the set $\ddot Z_\rho$ of boundary points of connected components of the open set $Z_\rho$ is at most countable. It is easy to see that each point $s\in\Omega_\rho\setminus(Z_\rho\cup\ddot Z_\rho)$ with $\rho(s)=0$ is $I$-positive.   

\begin{lemma}\label{l:tau} For any point $s\in \Omega_\rho\cap \dot\Omega_{\psi}\cap\dot\Omega_{\br'}\setminus\ddot Z_\rho$, the tangential curvature $\tau(s)$  can be calculated by the formula:
$$
\begin{aligned}
\tau(s)=\;&\rho(s)\cdot\int_0^s\psi'(v)\,dv-\frac{\int_0^{L}\int_0^u\rho(s){\cdot}\psi'(v)\,dv\,du}{\int_0^{L}\rho(u)\,du}, 
\end{aligned}
$$
where 
$$\psi'(s)=\begin{cases}
1&\mbox{if $s$ is $I$-null};\\
\smallskip
1+\lim\limits_{\e\to0}\dfrac{\|\bF(s+\e)-\bF(s-\e)\|-2\e}{\rho(s)\e^3}&\mbox{if  $\rho(s)>0$};\\
\smallskip
1+\lim_{\e\to+0}\dfrac{\|\br(s+\II^+_s[\e])-\br(s+\II^-_s[\e])\|-(\II^+_s[\e]-\II^-_s[\e])}{\JJ_s(\II^+[\e])-\JJ_s(\II^-_s[\e])}&\mbox{if $s$ is $I$-positive}.
\end{cases}
$$
\end{lemma}

\begin{proof} The formula for calculating $\psi'(s)$ follows from Lemmas~\ref{l:rho-positive}, \ref{l:I-null}, \ref{l:I-positive}. The local absolute continuity of the quotient curvature $\psi=\frac\Tau\Rho$ implies   
$
\psi(s)=\psi(0)+\int_{0}^s\psi'(v)\,dv
$
for any $s\in\IR$.
By Lemma~\ref{l:RT}, $$\tau(s)=\rho(s)\psi(s)=\rho(s)\psi(0)+\int_0^s\rho(s)\psi'(v)\,dv.$$

To find the value of $\psi(0)$, we apply Lemma~\ref{l:tau0} and obtain
$$0=\int_0^{L}\tau(u)\,du=\psi(0)\cdot \int_0^{L}\rho(u)\,du+\int_0^{L}\int_0^u\rho(u)\cdot\psi'(v)\,dv\,du$$
and finally
$$\psi(0)=-\frac{\int_0^{L}\int_0^u\rho(u){\cdot}\psi'(v)\,dv\,du}{\int_0^{L}\rho(u)du}.$$
\end{proof}

\section{Proof of Theorem~\ref{t:main}}\label{s:main}

Let $f:S_X\to S_Y$ be an isometry between the unit spheres of two absolutely smooth 2-dimensional Banach spaces $X,Y$. Fix any basis $\mathbf e_1,\mathbf e_2\in S_X$ for the Banach space $X$. 
 By the result of Tingley \cite{Tingley}, $f(-x)=-f(x)$ for any element $x\in S_X$. This implies that the vectors $f(\mathbf e_1),f(\mathbf e_2)$ form a basis in  the Banach space $Y$. Let $\br_X:\IR\to S_X$ and $\br_Y:\IR\to S_Y$ be the natural parameterizations of the unit spheres of the 2-based Banach spaces $X,Y$.
Lemmas~\ref{l:M} and \ref{l:iso} imply that $f\circ \br_X$ is a natural parameterization of the sphere $S_Y$ and hence $f\circ\br_X=\br_Y\circ\Phi$ for some isometry $\Phi$ of the real line. 
Taking into account that $f\circ\br_X(0)=f(\mathbf e_1)=\br_Y(0)$, we conclude that $\Phi(0)=0$ and hence $\Phi(x)=ax$ for some $a\in \{-1,1\}$. Let $L$ be the smallest positive real number such that $\br_X(L)=-\br_X(0)=-\mathbf e_1$.  Taking into account that $\mathbf e_2\in\mathbf r_X([0,L])$ and $f(\mathbf e_2)\in \mathbf r_Y([0,L])$, we conclude that $a=1$ and hence $\Phi$ is the identity isometry of the real line.
 
Therefore, $f\circ\mathbf r_X=\mathbf r_Y$. Taking into account that $f$ is an isometry, we can apply Lemmas~\ref{l:rho} and \ref{l:tau} and conclude that $\rho_X(s)=\rho_Y(s)$ and $\tau_X(s)=\tau_Y(s)$ for almost all $s\in\IR$. 
 
Take the linear isomorphism $F:X\to Y$ such that $F(\br_X(0))=\br_Y(0)$ and $F(\br'_X(0))=\br'_Y(0)$. 
Consider the curves $\br_Y:\IR\to Y$ and $\by=F\circ \br_X:\IR\to Y$ in the Banach space $Y$. 

By definitions of the curvatures $\rho_Y$ and $\tau_Y$, we have the equation 
$$\br_X''(s)=-\rho_Y(s)\br_Y(s)+\tau_Y(s)\br_Y(s)=-\rho_X(s)\br_Y(s)+\tau_X(s)\br'_Y(s)$$
holding for almost all $s\in\IR$.

Applying the linear operator $F$ to the equation
$$\br''_X(s)=-\rho_X(s){\cdot}\br_X(s)+\tau_X(s){\cdot}\br'_X(s),$$
we obtain the equation 
\begin{equation}\label{eq:bq}
\by''(s)=-\rho_X(s){\cdot}\by(s)+\tau_X(s){\cdot}\by'(s).
\end{equation}
Therefore, the $AC^1$-smooth functions $\br_Y$ and $\by$ satisfy the same differential equation (\ref{eq:bq}) and have the same initial positions:
$$
\by(0)=F\circ\br_X(0)=f(\mathbf e_1)=\br_Y(0)\quad\mbox{and}\quad
\by'(0)=F\circ\br'_X(0)=\br_Y'(0).
$$
Now the Uniqueness Theorem  \cite[65.2]{TP} for solutions of  linear differential equations of second order (see also the proof of Theorem~\ref{t:main2}) guarantees that $\by=\br_Y$ and hence  $$F\circ \br_X(s)=\by(s)=\br_Y(s)=f\circ\br_X(s)$$ for any $s\in \IR$. Therefore, the linear operator $F$ extends the isometry $f$. Since $F(S_X)=f(S_X)=S_Y$ we conclude that $F(B_X)=B_Y$ and hence $F:X\to Y$ is a linear isometry.

\section{Acknowledgements} The author would like to express his sincere thanks to Olesia Zavarzina (whose  interesting talk at the conference \cite{Zav}, \cite{KZ}  attracted the author's attention to Tingley's problem), to Vladimir Kadets for many inspiring discussions on this problem, and to the {\tt Mathoverflow} user Fedor Petrov for suggesting the idea\footnote{{\tt https://mathoverflow.net/a/344856/61536}} of the proof of Lemma~\ref{l:p}(2--4).

\end{document}